\documentclass[12pt]{amsart}
\usepackage[T1]{fontenc}
\usepackage[utf8]{inputenc}
\usepackage{color,amssymb,amsmath,mathrsfs,enumerate,esint}
\usepackage{a4wide}
\usepackage{graphicx}
\usepackage{pgf,tikz,pgfplots}
\usepackage{comment}
\usepackage{mathrsfs}
\usetikzlibrary{arrows}

\theoremstyle{plain}
\newtheorem{thm}{Theorem}[section]

\newtheorem{lemma}[thm]{Lemma}

\newtheorem{corollary}[thm]{Corollary}

\newtoks\prt
\newtheorem{proclaim}[thm]{\the\prt}
\theoremstyle{definition}

\newtheorem{remark}[thm]{Remark}

\newtheorem{definition}[thm]{Definition}

\def\eqn#1$$#2$${\begin{equation}\label#1#2\end{equation}}

\numberwithin{equation}{section}

\newtoks\by
\newtoks\paper
\newtoks\book
\newtoks\jour
\newtoks\yr
\newtoks\pages
\newtoks\vol
\newtoks\publ
\def\ota{{\hbox\vol{???}}}
\def\cLear{\by=\ota\paper=\ota\book=\ota\jour=\ota\yr=\ota
\pages=\ota\vol=\ota\publ=\ota}
\def\endpaper{\the\by, {\the\paper},
\textit{\the\jour} \textbf{\the\vol} (\the\yr), \the\pages.\cLear}
\def\endbook{\the\by, \textit{\the\book}, \the\publ.\cLear}
\def\endprep{\the\by, \textit{\the\paper}, \the\jour.\cLear}
\def\endyearprep{\the\by, \textit{\the\paper}, \the\jour, (\the\yr).\cLear}
\def\name#1#2{#2 #1}
\def\nom{ \rm no. }

\def\phi{\varphi}
\def\epsilon{\varepsilon}

\DeclareMathOperator*{\diam}{diam}
\DeclareMathOperator*{\dist}{dist}
\DeclareMathOperator*{\capacity}{cap}

\newcommand{\mN}{{\mathbb N}}
\newcommand{\en}{{\mathbb N}}
\newcommand{\R}{{\mathbb R}}
\newcommand{\rn}{{\mathbb R}^n}
\newcommand{\mR}{{\mathbb R}}
\newcommand{\mV}{{\mathbb V}}

\def\ve{\boldsymbol v}

\def\varnothing{\emptyset}
\def\f{\tilde{f}}

\newcommand{\mv}{{\mathbf v}}

\newcommand{\cL}{{\mathcal L}}
\newcommand{\cC}{{\mathcal C}}
\newcommand{\cH}{{\mathcal H}}

\title{Injectivity almost everywhere for weak limits of Sobolev homeomorphisms}

\author[O. Bouchala]{Ond\v{r}ej Bouchala}
\address{Department of Mathematical Analysis, Charles University, Sokolovsk\'{a}
  83, 186 00 Prague 8, Czech Republic}
\email{ondrej.bouchala@gmail.com}

\author[S. Hencl]{Stanislav Hencl}
\address{Department of Mathematical Analysis, Charles University, Sokolovsk\'{a}
  83, 186 00
  Prague 8, Czech Republic}
\email{hencl@karlin.mff.cuni.cz}

\author[A. Molchanova]{Anastasia Molchanova}
\address{Faculty of Mathematics, University of Vienna, Oskar-Morgenstern-Platz 1, 1090
Wien, Austria}
\email{anastasia.molchanova@univie.ac.at}

\thanks{OB and SH were supported by the grant GA\v{C}R P201/18-07996S.
AM was partially supported by Austrian Science Fund (FWF) projects M 2670 and F 65. OB was also partially supported the Charles University, project GA UK No. 960119.}

\begin{document}

\begin{abstract}
Let $\Omega\subset\rn$ be an open set and let $f\in W^{1,p}(\Omega,\rn)$ be a weak (sequential) limit of Sobolev homeomorphisms. Then $f$ is injective almost everywhere for $p>n-1$ both in the image and in the domain. For $p\leq n-1$ we construct a strong limit of homeomorphisms such that the preimage of a point is a continuum for every point in a set of positive measure in the image and a topological image of a point is a continuum for every point in a set of positive measure in the domain.
\end{abstract}

\maketitle

\section{Introduction}

Let $\Omega\subset\rn$ be an open set and let $f\colon\Omega\to\rn$ be a mapping.
In this paper, we study classes of mappings $f$ that might serve as deformations in Nonlinear Elasticity models. Following the pioneering papers of Ball \cite{Ball} and Ciarlet and Ne\v{c}as \cite{CN} we ask if our mapping is in some sense injective as the physical `non-interpenetration of the matter' asks a deformation to be
% in some sense
one-to-one.

There are several ways how to obtain injectivity or at least injectivity almost everywhere (\textit{a.e.}) of the mapping $f$.
As in \cite{Ball} we can ask that our mapping has finite energy where the energy functional $\int_{\Omega} W(Df)$ contains special terms
(like ratio of powers of $Df$, $\operatorname{adj} Df$ and $J_f$)
and any mapping with finite energy and reasonable boundary data is a homeomorphism (the reader is referred to e.g.\
\cite{HR,IS,MaVi} and \cite{Sv} for related results).

The approach motivated by Ball \cite{Ball} is fine if our mapping is continuous everywhere but in some %real physical
deformations the cavitation or even fractures may occur. To model these phenomena we need conditions which guarantee that our mapping is injective a.e.\ but on some small set bad things may happen.
Ciarlet and Ne\v{c}as \cite{CN} studied the class of mappings that satisfies
\begin{equation}\label{cond:Ciarlet-Necas}
\int_{\Omega} J_f\leq |f(\Omega)|
\end{equation}
together with $J_f>0$ a.e.\ and they showed that mappings of this class are injective a.e.\ in the image, see e.g.\
\cite{Ball2,BK,BHM,BeK,GKMS,MTY,T} for further results in this direction or \cite{KroVal2019,MieRou2016} for numerical treatment.
The inequality  \eqref{cond:Ciarlet-Necas}
is called the Ciarlet--Ne\v{c}as condition nowadays.
Note that the constraint $J_f>0$ a.e.\ is usually assumed in models of Nonlinear Elasticity as the `real deformation' cannot change its orientation and the energy density $W(Df(x))$ should tend to $\infty$ when $J_f(x)\to 0$, i.e.\ when we compress too much.

Another approach can be traced to M\"uller and Spector \cite{MS} where they studied a class of mappings that satisfy $J_f>0$ a.e.\ together with the (INV) condition
(see e.g.\  \cite{BHM,ConDeL2003,HMC,MST,SwaZie2002,SwaZie2004}). They showed that mappings in their class are one-to-one a.e.\ (see Section 5 for more information). Informally speaking, the (INV) condition means that the ball $B(x,r)$ is mapped inside the image of the sphere $f(S(a,r))$ and the complement $\Omega\setminus \overline{B(x,r)}$ is mapped outside $f(S(a,r))$ (see Preliminaries for the formal definition).

In all results in the previous paragraph the authors assume that $f\in W^{1,p}(\Omega)$ for some $p>n-1$.
We show that injectivity a.e.\ may fail horribly for $p\leq n-1$ even though the mapping $f$ is even a strong limit of homeomorphisms. We would like to stress that it fails even in the limiting case $p=n-1$ which is technically more involved. The class of mappings that we study in our project consists of weak (sequential) limits of Sobolev homeomorphisms. Homeomorphisms clearly satisfy the (INV) condition and so their weak limit must as well if $p>n-1$, since in this case the (INV) condition is closed under weak convergence (see \cite[Lemma 3.3]{MS}).
Therefore the class of weak limits of Sobolev homeomorphisms is a suitable class for variational models and one could expect that nice properties of homeomorphisms (like invertibility) could be carried to their weak limit.

The class of weak limits of Sobolev homeomorphisms was recently characterized in the planar case by Iwaniec and Onninen \cite{IO,IO2} and De Philippis and Pratelli \cite{DPP}. Moreover, one can study the orientation of mappings in this class \cite{HO} or even investigate planar BV weak limits and characterize their set of cavities and fractures \cite{CHKR}. In \cite{MV} Molchanova and Vodopyanov studied invertibility a.e.\ of a special subclass of weak limits of homeomorphisms. We generalize some of their results and we show the sharpness of the assumption $p>n-1$. Our first result is about the invertibility a.e.\ in the image. By continuum we mean the image of the segment $[0,1]$ in $\rn$ by a continuous one-to-one mapping. See Preliminaries for the definition of a precise representative of a Sobolev mapping.

\prt{Theorem}
\begin{proclaim}\label{firstthm}
Let $\Omega\subset\rn$ be open and let $f\colon \Omega\to \rn$ be a weak limit of Sobolev  homeomorphisms
$f_k\in W^{1,p}(\Omega,\rn)$, $p>n-1$ for $n>2$ or $p\geq 1$ for $n=2$. Then there is a precise representative $\hat{f}$ and a set $N_1\subset \rn$ of Hausdorff dimension $n-1$ such that the preimage $\hat{f}^{-1}(y)$ consists of only one point for every
$y\in \hat{f}(\Omega)\setminus N_1$.

On the other hand for every $n\geq 3$ there is a continuous mapping $f\colon[-1,1]^n\to [-1,1]^n$ with $J_f>0$ a.e.\ which is a strong limit of Sobolev homeomorphisms $f_k\in W^{1,n-1}([-1,1]^n,\rn)$ with $f_k(x)=x$ for $x\in\partial[-1,1]^n$ such that
$$
\text{ there is }C_A\subset [-1,1]^n \text{ with }|C_A|>0\text{ and }f^{-1}(y)\text{ is a continuum for every }y\in C_A.
$$
\end{proclaim}

Let us point out that the positive part of the statement essentially follows from the known results and techniques (\cite{BHM,MS,MST}) while the counterexample is entirely new and it is our main contribution. In the positive direction we only remove the assumption $J_f>0$ a.e.\ from \cite{MS} to have a mathematically complete theory. It is interesting  that the Hausdorff dimension of the critical set $N_1$ suddenly jumps from $n-1$ to $n$ as $p$ changes from $p>n-1$ to $p=n-1$. Note that the bound of dimension $n-1$ for $N_1$ for $p>n-1$ is sharp as the mapping
$[x_1,x_2,\hdots,x_n]\to[0,x_2,\hdots,x_n]$ shows. In \cite[Section 11]{MS} there is a counterexample (in case $p<n=2$), which shows that the weak limit of a sequence of one-to-one a.e.\ mappings might be two-to-one in a set of positive measure if (INV) is not satisfied. Our counterexample is entirely different as it is $\infty$-to-one and it is in some sense `monotone' as a strong limit of homeomorphisms, which is definitely not the case for a mapping from \cite{MS}.

Our second result is about the invertibility a.e.\ in the domain. See Preliminaries for the definition of the topological image $f^T(x)$.

\prt{Theorem}
\begin{proclaim}\label{secondthm}
Let $\Omega\subset\rn$ be open and let $f\colon\Omega\to \rn$ be a weak limit of Sobolev homeomorphisms
$f_k\in W^{1,p}(\Omega,\rn)$, $p>n-1$ for $n>2$ or $p\geq 1$ for $n=2$. Then there is a set $N_2\subset \rn$ of Hausdorff dimension $n-p$ such that the image $f^T(x)$ consists of only one point for every $x\in \Omega\setminus N_2$. If we moreover assume that $J_f>0$ a.e then there is a set $N_3$ of zero measure such that $f|_{\Omega\setminus N_3}$ is one-to-one.

On the other hand for every $n\geq 3$ there is $\tilde{f}\colon[-1,1]^n\to [-1,1]^n$ with $J_{\tilde{f}}>0$ a.e.\ which is a strong limit of Sobolev homeomorphisms $\f_k\in W^{1,n-1}([-1,1]^n,\rn)$ with $\f_k(x)=x$ for $x\in\partial [-1,1]^n$.
 The quasicontinuous representative of $\tilde{f}$ is one-to-one on $[-1,1]^n$ $($but $\tilde{f}([-1,1]^n)\subsetneq [-1,1]^n$$)$.
There is a continuous mapping $w\colon [-1,1]^n\to\rn$ which is a generalized inverse to
$\tilde{f}$, i.e.\ $w(\tilde{f}(x))=x$ for every $x\in [-1,1]^n$ such that
$$
\text{ there is } C_A\subset[-1,1]^n \text{ with } |C_A|>0 \text{ and } w^{-1}(x) \text{ is a continuum for every } x\in C_A.
$$

\end{proclaim}

Locally constant mapping shows that the assumption $J_f>0$ a.e.\ is needed for the conclusion that $f|_{\Omega\setminus N_3}$ is one-to-one. Moreover, there is no bound for the Hausdorff dimension of $N_3$ as there is a Lipschitz mapping $f$ which maps a set of dimension $n$ to a single point (see Example \ref{last} below).

As in Theorem \ref{firstthm} the positive result essentially follows from the known results (\cite{BHM,MS,MST}) while the counterexample is entirely new. As above the counterexample exists also for the critical exponent $p=n-1$ and there is again a sudden jump in the dimension of the critical set $N_2$ from $n-p\leq 1$ to $n$.

\section{Preliminaries}

By $B(c,r)$ we denote the euclidean ball with center $c\in\rn$ and radius $r>0$, and $S(c,r)$ stands for the corresponding sphere.

\subsection{Precise representative of a Sobolev mapping}\label{sec:PR}

Recall the following result from \cite[Theorem 3.3.3 and Theorem 2.6.16]{Z}.

\prt{Theorem}
\begin{proclaim}\label{repres}
Let $1\leq p\leq n$ and let $f\in W^{1,p}(\rn)$ be a $p$-quasicontinuous representative and set
$$
E_p=\{x\in \rn:\ x\text{ is not a Lebesgue point of }f\}\ .
$$
Then $\dim_{\mathcal{H}}(E_p)\leq n-p$. % and for $p=1$ we moreover get $\haus^{n-1}(E_1)=0$.
\end{proclaim}

We put
\begin{equation} \label{def:f*}
f^{\ast}(x) =
\begin{cases}
	\displaystyle{\lim_{r\to 0^{+}} \frac{1}{|B(x,r)|} \int_{B(x,r)} f(y) \, dy} & \text{if the limit exists}, \\
	0 & \text{otherwise}.
\end{cases}
\end{equation}
Note, that the representative $f^*$ is $p$-quasicontinuous (see remarks after
\cite[Proposition 2.8]{MS}).
We define a \textit{precise representative} of $f\in W^{1,p}(\Omega,\R^n)$ as any representative which is equal to $f^\ast$
up to a set of $p$-capacity $0$ (see e.g.\ \cite[Section 2.6]{Z} for the definition of capacity).

Here is a useful observation \cite[Lemma 2.9]{MS} about the representative $f^\ast$.
\begin{lemma}\label{lem:MS2.9}
	Let $f_k\to f$ weakly in $W^{1,p}(\Omega, \R^n)$, $a\in \Omega$ and $r_a:=\dist(a,\partial \Omega)$.
	Then there is an $\cL^1$-null set $N_a$ such that
	for any $r\in (0,r_a) \setminus N_a$ there exists a subsequence $f_j$ such that
	$f_j^* \to f^*$ weakly in $W^{1,p}(S(a,r), \R^n)$.
	Furthermore, if $p>n-1$ then
	$f_j^* \to f^*$ uniformly on $S(a,r)$.
\end{lemma}

\subsection{Topological degree}
Given a smooth map $f$ from $\Omega\subset\rn$ into $\rn$ we can define the {\it topological degree} as
$$\deg(f,\Omega,y_0)=\sum_{\{x\in\Omega: f(x)=y\}} \operatorname{sgn}(J_f(x))$$
if $J_f(x)\neq 0$ for each $x\in f^{-1}(y)$.
This definition can be extended to arbitrary continuous mappings and each point, see e.g.\ \cite{FG}.

The value of the degree of a continuous mapping $f\colon\overline{B(a,r)}\to\rn$ depends only on its values on the boundary $S(a,r)$. Thus, given a continuous mapping $f\colon S(a,r)\to\rn$ we use the notation $\deg(f, S(a,r),y)$ for $\deg(\hat{f}, B(a,r),y)$, where
$\hat{f}\colon\overline{B(a,r)}\to\rn$ is any continuous extension of $f\colon S(a,r)\to\rn$.

The degree is known to be stable under uniform convergence (see e.g.~\cite[Theorem 2.3 (1)]{FG}), i.e.
\eqn{stability}
$$
f_k\rightrightarrows f\text{ on }S(b,s)\text{ and }y\notin f(S(b,s))\Rightarrow
\lim_{k\to\infty} \deg(f_k, S(b,s),y)=\deg(f, S(b,s),y).
$$

It is also well-known that for a homeomorphism $f$ and $y\notin f(S(a,r))$ we have
\eqn{monotone}
$$
\deg(f, S(a,r),y)\neq 0\Leftrightarrow y\in B(a,r).
$$

\subsection{(INV) condition}

Suppose that
$f\colon S(a,r) \to \mathbb{R}^n$ is continuous,
following \cite{MS} we define a {\it topological image} of $B(a,r)$ as
$$
f^T(B(a,r)):=\bigl\{y\in \rn\setminus f(S(a,r)):\ \deg(f,S(a,r),y)\neq 0\bigr\}.
$$
Denote
$$E(f,B(a,r)):=f^T(B(a,r))\cup f(S(a,r)).$$

\begin{definition}[(INV) condition]
We say that $f\colon\Omega\to\R^n$ satisfies the condition (INV), provided that for every $a\in\Omega$ there exists an $\cL^1$-null set $N_a$ such that for all $r\in(0,\dist(a,\partial\Omega))\setminus N_a$ the mapping $f|_{S(a,r)}$ is continuous,
\begin{enumerate}[(i)]
	\item $f(x)\in f^T(B(a,r))\cup f(S(a,r))$ for $\cL^n$-a.e.\ $x\in \overline{B(a,r)}$ and
	\item $f(x)\in\R^n\setminus f^T(B(a,r))$ for $\cL^n$-a.e.\ $x\in\Omega\setminus B(a,r)$.
\end{enumerate}
%Here $r_a:=\dist(a,\partial\Omega)$.
\end{definition}

Moreover, we define the multifunction which describes the topological image $f^T(x)$ of a point as
$$
f^T(x) := \bigcap_{r>0,\ r\notin N_x} E(f^\ast, B(x,r)),
$$
where $f^\ast$ is given by~\eqref{def:f*}.
%the precise representative of $f$.
Let us recall that a quasicontinuous representative of $f\in W^{1,p}(\Omega,\rn)$, $p>n-1$, is continuous for every $x$ on almost every sphere $S(x,r)$.

\subsection{Cantor-set construction}\label{ssec:CS}

Following \cite[Section 4.3]{HK} we consider a Cantor-set construction in $(-1,1)^n$.

Denote the cube with center at
$a$ and edge $2r$ by
$Q(a,r) = (a_1 - r, a_1+r) \times \dots \times (a_n - r, a_n+r) $.
Let
$\mV$ be the set of
$2^n$
vertices of the cube
$[-1, 1]^n \subset \mR^n$
and
$\mV^k = \mV \times \cdots \times \mV$,
$k\in \mN$.
Consider a decreasing sequence
$\{\alpha_k\}_{k=0}^{\infty}$
such that $\alpha_k\approx \alpha_{k+1}$,
$1=\alpha_0\geq \alpha_1 \geq \dots >0$,
$$r_k = 2^{-k} \alpha_k\text{ and }r'_k=2^{-k}\alpha_{k-1}.$$
Set $z_0=0$, then $Q(z_0,r_0) = (-1,1)^n$
and we proceed by induction.
For
$$\mv(k)=(v_1, \dots , v_k) \in \mV^k$$
we denote
$$\mv(k-1)=(v_1, \dots , v_{k-1})$$
and define (see Fig.~\ref{pic:CantorSet})
\begin{equation*}
\begin{aligned}
    & z_{\mv(k)} = z_{\mv(k-1)} + \frac{1}{2} r_{k-1}v_k = z_0 + \frac{1}{2}\sum_{j=1}^{k} r_{j-1}v_j,\\
    & Q'_{\mv(k)} = Q(z_{\mv(k)}, r'_k) \quad \text{and} \quad  Q_{\mv(k)} = Q(z_{\mv(k)}, r_k).
\end{aligned}
\end{equation*}

\begin{figure}[h]
% \begin{center}
\unitlength=0.7mm
\begin{center}
\begin{picture}(110,50)(0,5)
%\circle{10}
\put(10,10){\framebox(40,40){}}
\put(30,10){\line(0,1){40}}
\put(10,30){\line(1,0){40}}
\put(12,12){\framebox(16,16){}}
%\put(20,32){\makebox(5,5){$R_{[-1,1]}$}}
\put(12,32){\framebox(16,16){}}
\put(32,12){\framebox(16,16){}}
\put(32,32){\framebox(16,16){}}
%end of first cubeanalogs of the computation above together with
\put(60,10){\framebox(40,40){}}
\put(80,10){\line(0,1){40}}
\put(60,30){\line(1,0){40}}
\put(62,12){\framebox(16,16){}}%1. rectangle
\put(70,12){\line(0,1){16}}
\put(62,20){\line(1,0){16}}
\put(63,13){\framebox(6,6){}}
\put(71,13){\framebox(6,6){}}
\put(63,21){\framebox(6,6){}}
\put(71,21){\framebox(6,6){}}
\put(62,32){\framebox(16,16){}}%2. rectangle
\put(70,32){\line(0,1){16}}
\put(62,40){\line(1,0){16}}
\put(63,33){\framebox(6,6){}}
\put(71,33){\framebox(6,6){}}
\put(63,41){\framebox(6,6){}}
\put(71,41){\framebox(6,6){}}
\put(82,12){\framebox(16,16){}}%3. rectangle
\put(90,12){\line(0,1){16}}
\put(82,20){\line(1,0){16}}
\put(83,13){\framebox(6,6){}}
\put(91,13){\framebox(6,6){}}
\put(83,21){\framebox(6,6){}}
\put(91,21){\framebox(6,6){}}
\put(82,32){\framebox(16,16){}}%4. rectangle
\put(90,32){\line(0,1){16}}
\put(82,40){\line(1,0){16}}
\put(83,33){\framebox(6,6){}}
\put(91,33){\framebox(6,6){}}
\put(83,41){\framebox(6,6){}}
\put(91,41){\framebox(6,6){}}
\end{picture}
% \centerline{{\bf Fig.~1} Cubes $Q_{\ve(k)}$ and $Q'_{\ve(k)}$ for $k=1,2$.}
% \vskip 10pt
\end{center}
\caption{Cubes $Q_{\ve(k)}$ and $Q'_{\ve(k)}$ for $k=1$, $2$.}
\label{pic:CantorSet}
\end{figure}
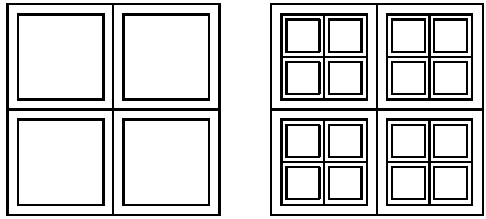

The measure of the $k$-th frame
$Q'_{\mv(k)} \setminus Q_{\mv(k)}$ is
\eqn{measure}
$$
    \cL^n(Q'_{\mv(k)} \setminus Q_{\mv(k)}) =  (2r_k')^n - (2r_k)^n
		\approx 2^{-nk}(\alpha_{k-1}-\alpha_k)\alpha_k^{n-1},
$$
and we have $2^{nk}$ such frames.

Denote
$A:=\{\alpha_k\}_{k=0}^{\infty}$,
the resulting Cantor set
$$
    C_A := \bigcap_{k=1}^{\infty} \bigcup_{\mv(k)\in \mV^{k}} Q_{\mv(k)}
$$
is a product of $n$ Cantor sets $\cC_\alpha$ in $\mR$
$$
    C_A = \cC_\alpha\times \dots \times \cC_\alpha,
$$
and the number of cubes in
$\{Q_{\mv(k)}: \mv(k) \in \mV^k\}$
is
$2^{nk}$.
Hence,
$$\cL^n(C_A)=\lim_{k\to\infty} 2^{nk}(2 \alpha_k 2^{-k})^n = \lim_{k\to\infty}2^n \alpha_k^n.$$

%=====================================================================
\subsection{Homeomorphism that maps a Cantor set onto another one}\label{sec:map_CS}

Consider two sequences
$A=\{\alpha_k\}_{k=0}^{\infty}$
and
$B=\{\beta_k\}_{k=0}^{\infty}$,
and two Cantor sets
$C_A$ and $C_B$
are designed according Section~\ref{ssec:CS}.
We also define
\begin{equation*}
\begin{aligned}
    %& r_k=2^{-k}a_k, \quad r_k' = 2^{-k}a_{k-1}, \\
    & \tilde{r}_k = 2^{-k}\beta_k, \quad \tilde{r}_k' = 2^{-k}\beta_{k-1},\\
    %& z_{\mv(k)} = z_{\mw(k)} + \frac{1}{2} r_{k-1}v_k = z_0 + \frac{1}{2}\sum_{j=1}^{k} r_{j-1}v_j,\\
    & \tilde{z}_{\mv(k)} = \tilde{z}_{\mv(k-1)} + \frac{1}{2} \tilde{r}_{k-1}v_k = \tilde{z}_0 + \frac{1}{2}\sum_{j=1}^{k} \tilde{r}_{j-1}v_j,\\
   % & Q'_{\mv(k)} = Q(z_{\mv(k)}, r_k'), \quad  Q_{\mv(k)} = Q(z_{\mv(k)}, r_{k}), \\
    & \tilde{Q}'_{\mv(k)} = Q(\tilde{z}_{\mv(k)}, \tilde{r}_{k}'), \quad  \tilde{Q}_{\mv(k)} = Q(\tilde{z}_{\mv(k)}, \tilde{r}_{k}).
\end{aligned}
\end{equation*}

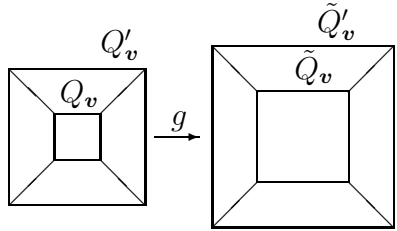
\begin{figure}[h]
\begin{center}
\unitlength=0.6mm
\begin{picture}(115,60)(0,5)
\put(10,20){\framebox(30,30){}}%domain
\put(20,30){\framebox(10,10){}}
%\put(15,15){\dashbox(20,40){}}
\put(10,20){\line(1,1){10}}
\put(40,20){\line(-1,1){10}}
\put(10,50){\line(1,-1){10}}
\put(40,50){\line(-1,-1){10}}
\put(32,53){\makebox(5,5){$Q_{\ve}'$}}
\put(23,42){\makebox(5,5){$Q_{\ve}$}}
%\put(13,32){\makebox(5,5){$A$}}
%\put(33,32){\makebox(5,5){$A$}}
%\put(23,12){\makebox(5,5){$B$}}
%\put(23,52){\makebox(5,5){$B$}}
\put(42,35){\vector(1,0){10}}%sipka
\put(45,36){\makebox(5,5){$g$}}
\put(55,15){\framebox(40,40){}}%target
\put(65,25){\framebox(20,20){}}
%\put(65,25){\dashbox(50,20){}}
\put(55,15){\line(1,1){10}}
\put(95,15){\line(-1,1){10}}
\put(55,55){\line(1,-1){10}}
\put(95,55){\line(-1,-1){10}}
\put(80,58){\makebox(5,5){$\tilde{Q}'_{\ve}$}}
\put(75,48){\makebox(5,5){$\tilde{Q}_{\ve}$}}
\end{picture}
\end{center}
\caption{The transformation of $Q'_v\setminus Q_v$
onto
$\hat{Q}'_v\setminus \hat{Q}_v$ for $n=2$}
\label{pic:g}
\end{figure}

There exists a homeomorphism
$g$
which maps $C_A$ onto $C_B$ (see Fig.~\ref{pic:g}).
Moreover, in $Q'_{\mv(k)} \setminus Q_{\mv(k)}$ we have analogously to \cite[proof of Theorem 4.10]{HK}
\begin{equation}\label{eq:Dg}
    |Dg(x)| \approx
    \max \left\{\frac{\tilde{r}_k}{r_k}, \frac{\frac{\tilde{r}_{k-1}}{2} - \tilde{r}_{k}}{\frac{r_{k-1}}{2} - r_{k}}\right\} =
    \max \left\{\frac{\beta_k}{\alpha_k}, \frac{\beta_{k-1} -\beta_{k}}{\alpha_{k-1} - \alpha_{k}}\right\}
\end{equation}
and
$$
    J_g(x) \sim
    \frac{\frac{\tilde{r}_{k-1}}{2} - \tilde{r}_{k}}{\frac{r_{k-1}}{2} - r_{k}}\left(\frac{\tilde{r}_k}{r_k}\right)^{n-1}.
$$
Likewise, for $y\in \tilde{Q}'_{\mv(k)} \setminus \tilde{Q}_{\mv(k)}$ we have
\eqn{eq:Dg2}
$$
|Dg^{-1}(y)|\approx\max \left\{\frac{\alpha_k}{\beta_k},
\frac{\alpha_{k-1} - \alpha_{k}}{\beta_{k-1} -\beta_{k}}\right\}
\text{ and }
J_{g^{-1}}(y) \sim
\frac{\frac{r_{k-1}}{2} - r_{k}}{\frac{\tilde{r}_{k-1}}{2} - \tilde{r}_{k}}
		\left(\frac{r_k}{\tilde{r}_k}\right)^{n-1}.
$$

More precisely we define this $g$ as a uniform limit of bilipschitz mappings $g_k$ which map the $k$-th iteration of the Cantor set $C_A$ onto the $k$-th iteration of $C_B$. That is
\eqn{star}
$$
g_{k}(x)=
g(x)\text{ for }x\notin \bigcup_{\mv(k) \in \mV^k} Q_{\mv(k)}
$$
and
$$
g_{k}\text{ maps }Q_{\mv(k)}\text{ onto }\tilde{Q}_{\mv(k)}\text{ linearly for } \mv(k) \in \mV^k.
$$

%=====================================================================

\subsection{Constructing a Cantor tower}\label{map_CT}

We build a Cantor tower as in \cite{GHT}.

Suppose $n \ge 2$ and denote by $\hat{\mathbb{V}}$ the set of points
%\begin{align*}
%\bigl(0,0, \ldots, 0, -1+\tfrac{2j}{2^{n}} \bigr)
%\end{align*}
%\textcolor{blue}{Shouldn't it be
\begin{align*}
\bigl(0,0, \ldots, 0, -1+\tfrac{2j-1}{2^{n}} \bigr)
\end{align*}
%here?
%Because in the case $n=2$ and $k=1$ we have $\hat{z}_{\ve(1)}=\{-\frac{1}{2},0,\frac{1}{2},1\}$.}
where $j=1,2, \ldots, 2^{n}$. Sets
\begin{align*}
\hat{\mathbb{V}}^{k} := \hat{\mathbb{V}} \times \cdots \times \hat{\mathbb{V}}, \quad k \in \mathbb{N},
\end{align*}
serve as sets of indices in the construction of a Cantor tower.

Suppose that $\{\beta_{k} \}_{k=0}^{\infty}$ is a decreasing sequence as before with $1 = \beta_{0}$ and $\beta_{i}>2^n \beta_{i+1}$, and define
\eqn{defhatrk}
$$
\hat{r}_{k} := 2^{-k}\beta_{k}\text{ and }\hat{r}'_{k} := 2^{-k}\beta_{k-1}.
$$
Set $\hat{z}_{0} = 0$. Then it follows that $Q(\hat{z}_{0}, \hat{r}_{0}) = (-1,1)^{n}$ and we proceed further by induction. For $\hat{\ve}(k) := (\hat{v}_{1}, \hat{v}_{2}, \ldots, \hat{v}_{k}) \in \hat{\mathbb{V}}^{k}$ we denote
$\hat{\ve}(k-1) := (\hat{v}_{1}, \hat{v}_{2}, \ldots, \hat{v}_{k-1})$ and define (see Fig.~3)
\eqn{qqq}
$$
\begin{aligned}
&\hat{z}_{\hat{\ve}(k)} := \hat{z}_{\hat{\ve}(k-1)} + \hat{r}_{k-1}\hat{v}_{k} = \hat{z}_{0} + \sum_{j=1}^{k} \hat{r}_{j-1}\hat{v}_{j}\\
&\hat{Q}_{\hat{\ve}(k)}' := Q(\hat{z}_{\hat{\ve}(k)}, \hat{r}'_{k}) \text{ and } \hat{Q}_{\hat{\ve}(k)} := Q(\hat{z}_{\hat{\ve}(k)}, \hat{r}_{k})
\end{aligned}
$$

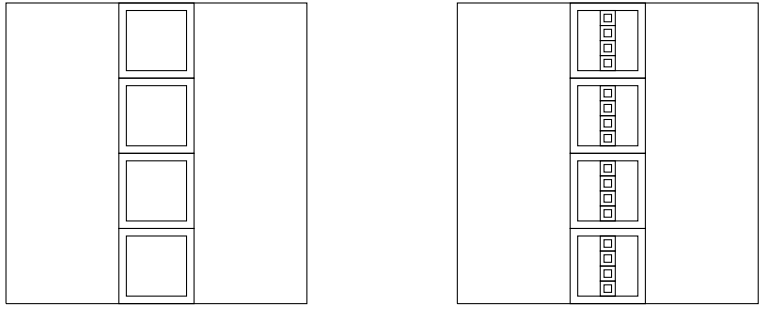
\begin{figure}[ht]
\begin{center}
\begin{tikzpicture}[scale=0.50]
%-----THE BIG CUBE 1

\draw (-4,-4)--(4,-4)--(4,4)--(-4,4)--(-4,-4);
%\draw (-4.3,-4.3) rectangle (4.3,4.3);

\draw (-1,-4) rectangle (1,-2);
\draw (-0.8,-3.8) rectangle (0.8,-2.2);

\draw (-1,-2) rectangle (1,0);
\draw (-0.8,-1.8) rectangle (0.8,-0.2);

\draw (-1,0) rectangle (1,2);
\draw (-0.8,0.2) rectangle (0.8,1.8);

\draw (-1,2) rectangle (1,4);
\draw (-0.8,2.2) rectangle (0.8,3.8);

%--------------THE BIG CUBE 2

\draw (8,-4)--(16,-4)--(16,4)--(8,4)--(8,-4);
%\draw (7.8,-4.2)--(16.2,-4.2)--(16.2,4.2)--(7.8,4.2)--(7.8,-4.2);

\draw (11,-4) rectangle (13,-2);
\draw (11.2,-3.8) rectangle (12.8,-2.2);

\draw (11.8,-3.8) rectangle (12.2,-3.4);
\draw (11.9,-3.7) rectangle (12.1,-3.5);

\draw (11.8,-3.4) rectangle (12.2,-3);
\draw (11.9,-3.3) rectangle (12.1,-3.1);

\draw (11.8,-3) rectangle (12.2,-2.6);
\draw (11.9,-2.9) rectangle (12.1,-2.7);

\draw (11.8,-2.6) rectangle (12.2,-2.2);
\draw (11.9,-2.5) rectangle (12.1,-2.3);

%-------------

\draw (11,-2) rectangle (13,0);
\draw (11.2,-1.8) rectangle (12.8,-0.2);

\draw (11.8,-1.8) rectangle (12.2,-1.4);
\draw (11.9,-1.7) rectangle (12.1,-1.5);

\draw (11.8,-1.4) rectangle (12.2,-1);
\draw (11.9,-1.3) rectangle (12.1,-1.1);

\draw (11.8,-1) rectangle (12.2,-0.6);
\draw (11.9,-0.9) rectangle (12.1,-0.7);

\draw (11.8,-0.6) rectangle (12.2,-0.2);
\draw (11.9,-0.5) rectangle (12.1,-0.3);

%----------------------
\draw (11,0) rectangle (13,2);
\draw (11.2,0.2) rectangle (12.8,1.8);

\draw (11.8,0.2) rectangle (12.2,0.6);
\draw (11.9,0.3) rectangle (12.1,0.5);

\draw (11.8,0.6) rectangle (12.2,1);
\draw (11.9,0.7) rectangle (12.1,0.9);

\draw (11.8,1) rectangle (12.2,1.4);
\draw (11.9,1.1) rectangle (12.1,1.3);

\draw (11.8,1.4) rectangle (12.2,1.8);
\draw (11.9,1.5) rectangle (12.1,1.7);

%-----------

\draw (11,2) rectangle (13,4);
\draw (11.2,2.2) rectangle (12.8,3.8);

\draw (11.8,2.2) rectangle (12.2,2.6);
\draw (11.9,2.3) rectangle (12.1,2.5);

\draw (11.8,2.6) rectangle (12.2,3);
\draw (11.9,2.7) rectangle (12.1,2.9);

\draw (11.8,3) rectangle (12.2,3.4);
\draw (11.9,3.1) rectangle (12.1,3.3);

\draw (11.8,3.4) rectangle (12.2,3.8);
\draw (11.9,3.5) rectangle (12.1,3.7);
%----------------------------------------

\end{tikzpicture}
\end{center}

% \centerline{{\bf Fig.~3} Cubes $\hat{Q}_{\ve(k)}$ and $\hat{Q}'_{\ve(k)}$ for $k=1,2$ in the construction of the Cantor's tower.}
% \vskip 10pt
\caption{Cubes $\hat{Q}_{\hat{\ve}(k)}$ and $\hat{Q}'_{\hat{\ve}(k)}$ for $k=1$, $2$ in the construction of the Cantor's tower.}
\end{figure}

\subsection{Bilipschitz mapping which takes a Cantor set onto a Cantor's tower}\label{sec:BiLip}

Let us now define the Cantor set $C_{B}$ as in Section~\ref{ssec:CS} by choosing
\eqn{defbeta}
$$\beta_{k} = 2^{-k\beta},$$
where $\beta \geq n+1$. Using this sequence we also define the Cantor tower $C_{B}^{T}$ as in Section~\ref{map_CT}.
%\textcolor{blue}{Isn't here $c_{k} = 2^{-k (\beta-n)}$?
%Because cubes ${Q}_{\ve(k)}=\hat{Q}_{\ve(k)}$, right?}
As $\beta\geq n+1$, we see that
%\eqn{ahoj}
$$
\hat{Q}_{\hat{\ve}(k)}=Q(\hat{z}_{\hat{\ve}(k)}, 2^{-k}\beta_{k})\subsetneq
Q(\hat{z}_{\hat{\ve}(k)}, 2^{-1-k}\beta_{k-1})=\frac{1}{2}\hat{Q}_{\hat{\ve}(k)}'
$$
and thus we have enough empty space in $\hat{Q}_{\hat{\ve}(k)}'\setminus \hat{Q}_{\hat{\ve}(k)}$ to move the cubes of the next generation into a tower formation.

The following theorem from \cite[Proposition 2.4]{GHT} gives us a bilipschitz mapping $L \colon \R^{n} \to \R^{n}$ which maps the Cantor set $C_{B}$ onto the Cantor tower $C_{B}^{T}$. We refer to this mapping as a \textit{tower mapping}.

%BLAA BLAA BLAA... If we now apply Arzel\`{a}-Ascoli theorem to both sequences $\{ L_{j} \}_{j=1}^{\infty}$ and $\{ L_{j}^{-1} \}_{j=1}^{\infty}$ we may find out that $\{ L_{j} \}_{j=1}^{\infty}$ converges uniformly to a bilipschitz map $L \colon (-1,1)^{n} \to (-1,1)^{n}$ such that
%\begin{align*}
%K^{-1}\abs{x-y} \le \abs{L(x)-L(y)}\le K\abs{x-y}
%\end{align*}
%for all $x,y \in (-1,1)^{n}$.

\prt{Theorem}
\begin{proclaim}\label{Bilip}
Suppose that $C_{B}$ is the Cantor set and $C_{B}^{T}$ is the Cantor tower in $\R^{n}$ defined by the sequence
\begin{align*}
\beta_{k}= 2^{-k\beta} \, ,
\end{align*}
where $\beta\geq n+1$. Then there is a bilipschitz mapping $L \colon \R^{n} \to \R^{n}$ which takes $C_{B}$ onto $C_{B}^{T}$.
Moreover,
\eqn{goodmap}
$$
\text{ for every }\hat{\ve}(i)\in\hat{\mathbb{V}}^i\quad L^{-1}(\hat{Q}_{\hat\ve(i)})=\tilde{Q}_{\ve(i)}\text{ for some }\ve(i)\in \mathbb{V}^i.
$$
\end{proclaim}

%There is a bilipschitz mapping with a Lipschitz constant
%$L_{BC}$, which takes a Cantor set
%$C_B$
%onto a Cantor's tower
%$C_C^T$,
%if the parameters are small enough.
%Let's say
%$b_k = c_k = 2^{-k\beta}$
%for
%$\beta \geq n+1$.
%We need that to have enough place for the tower.

%\newpage

%=====================================================================

\subsection{Piecewise linear mappings}\label{sec:tentacle}

We define an auxiliary piecewise linear mapping and estimate its derivative.

Let
\begin{equation}%\label{coefficient_relations}
t_1<t_2<t_3<t_4\text{ and }s_1< s_2< s_3< s_4.
\end{equation}
We consider a piecewise linear mapping $h^{1}\colon [t_1,t_4] \to \R$ with $h(t_i)=s_i$, $i=1,2,3,4$, i.e.
%A mapping $h^{1}(\cdot) \colon [0, c] \to [0,\tilde{c}]$ acts as
\eqn{defh}
$$
    h(t; [t_1,s_1],[t_2,s_2],[t_3,s_3],[t_4,s_4]) =
    \begin{cases}
        \frac{s_2-s_1}{t_2-t_1}(t-t_1)+s_1,  & \text{ if } t_1\leq t \leq t_2, \\
      \frac{s_3-s_2}{t_3-t_2}(t-t_2)+s_2,  & \text{ if } t_2< t \leq t_3, \\
      \frac{s_4-s_3}{t_4-t_3}(t-t_3)+s_3,  & \text{ if } t_3< t \leq t_4. \\
    \end{cases}
$$
%Note, that
%$$
% (h)^{-1}(t; [t_1,s_1],[t_2,s_2],[t_3,s_3],[t_4,s_4])= h(t; [s_1,t_1],[s_2,t_2],[s_3,t_3],[s_4,t_4]).
%$$
Clearly
\eqn{estder}
$$|Dh(t)|=
\begin{cases}
    \frac{s_2-s_1}{t_2-t_1},  & \text{ if } t_1< t < t_2, \\
      \frac{s_3-s_2}{t_3-t_2},  & \text{ if } t_2< t < t_3, \\
      \frac{s_4-s_3}{t_4-t_3},  & \text{ if } t_3< t < t_4. \\
\end{cases}
%\text{ and }
%|D(h)^{-1}(s)|=
%\begin{cases}
%%    \frac{t_2-t_1}{s_2-s_1},  & \text{ if } s_1< s < s_2, \\
%      \frac{t_3-t_2}{s_3-s_2},  & \text{ if } s_2< s < s_3, \\
%      \frac{t_4-t_3}{s_4-s_3},  & \text{ if } s_3< s < s_4. \\
%\end{cases}
$$

\section{Injectivity in the image: counterexample in Theorem \ref{firstthm}}

\subsection{Definition of tentacles}\label{tentacles}

We start with a Cantor tower $C_B^T$ and for each point $y\in C_A$ we find a corresponding point $x\in C_B^T$ (see \eqref{corespond}).
We want to have a continuum $l_x$ (with the end point $x$) which goes onto $y$ by our mapping. For better visualization we first map $C_B^T$ on itself to squeeze this $l_x$ onto $x$ by mapping $h$. Then, with the help of a bilipschitz mapping $L^{-1}$ (Theorem \ref{Bilip}) we transform $C^T_B$ to $C_B$ and finally we map homeomorphically $C_B$ onto $C_A$ by $g^{-1}$ (see Subsection~\ref{sec:map_CS}),
i.e.\
% We start with a Cantor type set of positive measure $C_A$ and for each point $x\in C_A$ we want to have a continuum
% $l_x$ (with an end point $x$) which is mapped onto $x$ by our mapping. For a better visualization we first map
% $C_A$ homeomorphically by $g$ (see subsection \ref{sec:map_CS}) onto $C_B$ and then shift $C_B$ by bilipschitz mapping $L$  onto tower formation $C^T_B$ as in Theorem \ref{Bilip}.
% For each $x\in C_A$ we find corresponding $y=L(g(x))\in C_B^T$ and we define the continuum $l_y$ which is the corresponding image $L(g(l_x))$.
% We aim to squeeze $l_x$ onto $x$ by mapping $h$ and
the final mapping
$$
f=g^{-1}\circ L^{-1}\circ h
$$
squeezes $l_x$ onto $y$ (see Fig.~\ref{pic:the_mapping}).

For any $x\in C_B^T$ we find sequence $\hat{\ve}(k)\in \hat{\mathbb{V}}^{k}$ such that
\eqn{corespond}
$$
x=\bigcap_{k=1}^{\infty}\hat{Q}_{\hat{\ve}(k)}\text{ and this corresponds to }y=\bigcap_{k=1}^{\infty}Q_{\ve(k)}\in C_A,
$$
where the mapping $\hat{\ve}(k)\to\ve(k)$ is given by $L^{-1}(\hat{Q}_{\hat{\ve}_k})=\tilde{Q}_{\ve(k)}$ from \eqref{goodmap}.
Now for each $\hat{Q}_{\hat{\ve}(k)}$ we define a tentacle $T_{\hat{\ve}(k)}$ (a long and thin polyhedron) which contains $\hat{Q}_{\hat{\ve}(k)}$ and we set
\begin{equation}\label{def:l}
    l_x:=\bigcap_{k=1}^{\infty}T_{\hat{\ve}(k)}.
\end{equation}
First we define a `straight' tentacle $T^S_{\hat{\ve}(k)}$ and then we adjust it in the next subsection so that
$$
T_{\hat{\ve}(k+1)}\subset T_{\hat{\ve}(k)}\text{ whenever }\hat{\ve}(k+1)\text{ is a continuation of }\hat{\ve}(k),
$$
i.e.\ first $k$ terms of $\hat{\ve}(k+1)$ are exactly $\hat{\ve}(k)$ (see Fig.~\ref{pic:two_gen}).

Take the parameter $\beta$ from \eqref{defbeta}
% the definition of $\beta_k=2^{-\beta k}$
and recall \eqref{defhatrk}, that is $\hat{r}_k=2^{-k}\beta_k=2^{-k(\beta+1)}$.
%We fix large enough $\alpha\in\en$ whose exact value we specify later.
We define for $k\in \mathbb{N}$
% Note that $a^j_k$ is decreasing as a function of $j$ and increasing as a function of $k$.
% Similarly we define
$$
    a_k  = 1 -\sum_{i=0}^{k} \hat{r}_{i+2} \approx 1,\quad
    c_k  = 1 -\sum_{i=0}^{k-1} \hat{r}_{i+2} \approx 1,
$$
and further we fix decreasing sequences  $0<b_{k+1}<b_k<\frac{1}{e}$ and  $0<d_{k+1}<d_k<\frac{1}{e}$
\eqn{choicebd}
$$
\text{ such that }b_k<d_k<a_k<c_k\text{ and }d_{k+1}< 4^n b_k
$$
whose exact values we find by induction using Lemma \ref{thmH} below.

For $r>0$ and $\rho_1<\rho_2$ we define a parallelepiped
$$
P(r,\varrho_1, \varrho_2): = [\varrho_1,\varrho_2)\times(-r,r)\times \cdots \times  (-r,r).
$$
For each $k$ we also define
$$
P'_k:=P(d_k,\hat{r}_k,c_k)
\quad \text{and} \quad
P_k:=P(b_k,\hat{r}_k,a_k).
$$

Now we define `straight' tentacles as
$$
\begin{aligned}
T'^S_k:=Q(0,\hat{r}_k)\cup P'_k
& \quad \text{and} \quad
T^S_k:=Q(0,\hat{r}_k)\cup P_k, \\
T'^S_{\hat\ve(k)}:=\hat{z}_{\hat\ve(k)}+T'^S_k
& \quad \text{and} \quad
T^S_{\hat\ve(k)}:=\hat{z}_{\hat\ve(k)}+T^S_k. %\textcolor{blue}{\cap T'^S_{\hat\ve(k)}?}
\end{aligned}
$$
%(see Fig.~????).
Both $T'^S_k$ and $T^S_k$ clearly contain $Q(0,\hat{r}_k)$ and note that $T^S_k\subset T'^S_k$ as $c_k>a_k$ and $d_k>b_k$.
Moreover, $P'_k$ and $Q(0,\hat{r}_k)$ have one common side and thus $T'^S_k$ is connected.
Furthermore, the length of each tentacle $T'^S_k$ is bigger than $a_k>1 - \frac{1}{1 -  2^{-\beta-1}}$ and hence
\begin{equation}\label{def:ls}
l^S:=\bigcap_{k=1}^{\infty}T^S_{k}\text{ is a nontrivial segment}.
% l^S_x:=\bigcap_{k=1}^{\infty}T^S_{\hat{\ve}(k)}\text{ is a nontrivial segment for each }x.
\end{equation}
% \textcolor{red}{[It is not true, since $T^S_{\hat{\ve}(k)} \not\subset T^S_{\hat{\ve}(k-1)}$. Should it be for $x=0$?]}

%\ref{pic:tentacle}).
% A line parallel to $Ox_n$, and going through $\tilde{z}_{\mv(k)}$, will be a center-line of the tentacle.
% Then in each intersection of this line and $\partial\hat{Q}_j$ we take an $(n-1)$-dimensional cube with radius
% $r=\hat{r}_{(2+\alpha)k-j} = 2^{-(2+\alpha)k\beta+j\beta}$, and join this intersections by linear hyperplanes.

%\begin{figure}[h]\label{pic:two_gen}
%\center{\includegraphics[width=1\linewidth]{two_generations.jpg}}
%\caption{Two generations of tentacles.}
%\end{figure}

\begin{figure}[h]
\begin{center}
\begin{tikzpicture}[line cap=round,line join=round,>=triangle 45,x=0.15cm,y=0.15cm,scale=0.8]
\clip(0,-22) rectangle (140,22);
\draw (20,-20)-- (20,20);
\draw (20,20)-- (-20,20);
% \draw (-20,20)-- (-20,-20);
\draw (-20,-20)-- (20,-20);
\draw (3,12)-- (3,18);
\draw (3,18)-- (-3,18);
% \draw (-3,18)-- (-3,12);
\draw (-3,12)-- (3,12);
\draw (3,2)-- (3,8);
\draw (3,8)-- (-3,8);
% \draw (-3,8)-- (-3,2);
\draw (-3,2)-- (3,2);
\draw (3,-8)-- (3,-2);
\draw (3,-2)-- (-3,-2);
% \draw (-3,-2)-- (-3,-8);
\draw (-3,-8)-- (3,-8);
\draw (3,-18)-- (3,-12);
\draw (3,-12)-- (-3,-12);
% \draw (-3,-12)-- (-3,-18);
\draw (-3,-18)-- (3,-18);
\draw (20,8)-- (130,8);
% \draw (40,6)-- (65,4);
% \draw (65,4)-- (100,2);
\draw (130,8)-- (130,-8);
\draw (130,-8)-- (20,-8);
% \draw (65,-4)-- (40,-6);
% \draw (40,-6)-- (20,-10);
% \draw (20,-10)-- (20,10);
\draw (3,16.5)-- (20,7);
\draw (3,13.5)-- (20,4);
\draw (20,7)-- (130,7);
\draw (20,4)-- (130,4);
% \draw (100,1.45)-- (65,2.9);
% \draw (100,1.55)-- (65,3.1);
% \draw (40,5.1)-- (65,3.1);
% \draw (65,2.9)-- (40,3.9);
% \draw (100,4)-- (100,7);
\draw [dash pattern=on 2pt off 2pt] (3,16.5)-- (130,16.5);
\draw [dash pattern=on 2pt off 2pt] (3,13.5)-- (130,13.5);
\draw [dash pattern=on 2pt off 2pt] (130,16.5)-- (130,13.5);
% \draw [dash pattern=on 2pt off 2pt] (20,16.1)-- (40,15.6);
% \draw [dash pattern=on 2pt off 2pt] (20,13.9)-- (40,14.4);
% \draw [dash pattern=on 2pt off 2pt] (100,14.95)-- (65,14.9);
% \draw [dash pattern=on 2pt off 2pt] (100,15.05)-- (65,15.1);
% \draw [dash pattern=on 2pt off 2pt] (65,15.1)-- (40,15.6);
% \draw [dash pattern=on 2pt off 2pt] (65,14.9)-- (40,14.4);
% \draw [dash pattern=on 2pt off 2pt] (100,16.5)-- (100,14.95);
\draw [->,line width=0.5pt] (90,13) -- (90,7.5);
\draw[color=black] (72,0) node {\scriptsize $T^S_{k-1}$};
\draw[color=black] (133,5.5) node {\scriptsize $T'_k$};
\draw[color=black] (133,15) node {\scriptsize $T'^S_k$};
\draw[color=black] (92,11) node {\scriptsize $S_k$};
\end{tikzpicture}
\end{center}
\caption{Two generations of tentacles.}
\label{pic:two_gen}
\end{figure}
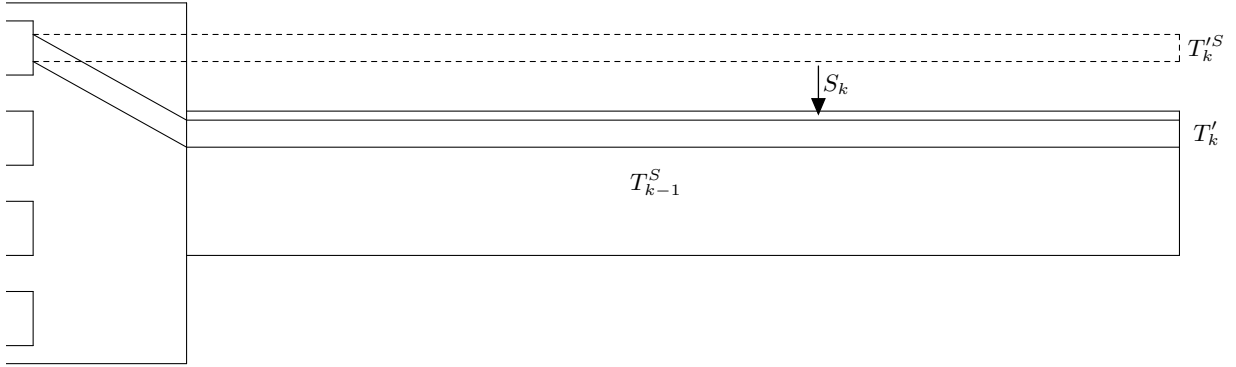

Let us estimate
\begin{equation}\label{meas1}
|P_k | \approx a_k \cdot (2b_k)^{n-1} \approx b_k^{n-1}\text{ and }
|P'_k|\approx c_k \cdot (2d_k)^{n-1} \approx d_k^{n-1}.
\end{equation}
%Similarly for each $j-1, \dots, k$ we have
%\begin{equation}\label{meas2}
%%|T'^S_k \setminus T^S_k|
%\leq
%(c_k d_k ^{n-1} - a_k b_k ^{n-1}) \approx 2^{-\alpha k(\beta+1) (n-1)}.
%\end{equation}

\subsection{Shifting of tentacles into previous tentacles}
In this section we want to shift the `straight' tentacles into `real' tentacles $T_{\hat{\ve}(k)}$ so that
\eqn{subset}
$$
T'_{\hat{\ve}(k+1)}\subset T_{\hat{\ve}(k)}\text{ whenever }\hat{\ve}(k+1)\text{ is a continuation of }\hat{\ve}(k).
$$

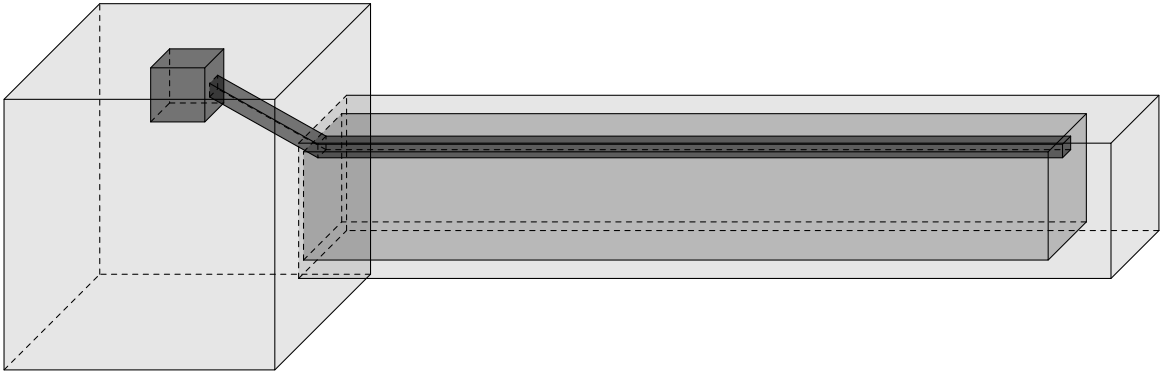
\begin{figure}[h]
\begin{center}
\begin{tikzpicture}[line cap=round,line join=round,>=triangle 45,x=0.18cm,y=0.18cm]
\clip(-15,-15) rectangle (75,15);
\fill[fill=black,fill opacity=0.1] (-13.53,-13.53) -- (-13.53,6.47) -- (-6.46,13.54) -- (13.54,13.54) -- (13.54,-6.46) -- (6.47,-13.53) -- cycle;
\draw (-13.53,-13.53) -- (-13.53,6.47) -- (6.47,6.47) -- (6.47,-13.53) -- cycle;
\draw [dash pattern=on 2pt off 2pt] (-6.46,-6.46) -- (-6.46,13.54);
\draw (-6.46,13.54) -- (13.54,13.54);
\draw (13.54,13.54)-- (13.54,-6.46);
\draw [dash pattern=on 2pt off 2pt] (13.54,-6.46)-- (-6.46,-6.46);
\draw (6.47,-13.53)-- (13.54,-6.46);
\draw (6.47,6.47)-- (13.54,13.54);
\draw (-6.46,13.54)-- (-13.53,6.47);
\draw [dash pattern=on 2pt off 2pt] (-13.53,-13.53)-- (-6.46,-6.46);
\fill[fill=black,fill opacity=0.5] (-1.3,10.2) -- (2.7,10.2) -- (2.7,6.2) -- (1.3,4.8) -- (-2.7,4.8) -- (-2.7,8.8) -- cycle;
\draw (2.7,10.2) -- (-1.3,10.2);
\draw [dash pattern=on 2pt off 2pt] (-1.3,10.2) -- (-1.3,6.2);
\draw [dash pattern=on 2pt off 2pt] (-1.3,6.2) -- (2.7,6.2);
% \draw [dash pattern=on 2pt off 2pt] (2.7,6.2) -- (2.7,6.2);
\draw (2.7,6.2) -- (2.7,10.2);
% \fill[fill=black,fill opacity=0.1] (2.7,10.2) -- (-1.3,10.2) -- (-1.3,6.2) -- (2.7,6.2) -- cycle;
\draw (1.3,8.8) -- (-2.7,8.8) -- (-2.7,4.8) -- (1.3,4.8) -- cycle;
\draw (-1.3,10.2)-- (-2.7,8.8);
\draw (2.7,10.2)-- (1.3,8.8);
\draw (2.7,6.2)-- (1.3,4.8);
\draw [dash pattern=on 2pt off 2pt] (-2.7,4.8)-- (-1.3,6.2);
\fill[fill=black,fill opacity=0.5] (1.65,7.65) -- (2.25,8.25) -- (10.25,3.75) -- (65.25,3.75) -- (65.25,2.75)  -- (64.65,2.15) -- (9.65,2.15) -- (1.65,6.65) -- cycle;
\draw (2.25,8.25) -- (10.25,3.75);
\draw (10.25,3.75) -- (65.25,3.75);
\draw (65.25,3.75) -- (65.25,2.75);
\draw [dash pattern=on 2pt off 2pt] (65.25,2.75) -- (10.25,2.75);
\draw [dash pattern=on 2pt off 2pt] (10.25,2.75) -- (2.25,7.25);
\draw [dash pattern=on 2pt off 2pt] (10.25,2.75) -- (10.25,3.75);
\draw [dash pattern=on 2pt off 2pt] (2.25,7.25) -- (2.25,8.25);
\draw [dash pattern=on 2pt off 2pt] (10.25,2.75) -- (9.65,2.15);
\draw [dash pattern=on 2pt off 2pt] (9.65,3.15) -- (9.65,2.15);
\draw [dash pattern=on 2pt off 2pt] (10.25,3.75) -- (9.65,3.15);
\draw [dash pattern=on 2pt off 2pt] (2.25,8.25)-- (1.65,7.65);
\draw [dash pattern=on 2pt off 2pt] (1.65,6.65)-- (2.25,7.25);
\draw [dash pattern=on 2pt off 2pt] (1.65,6.65)-- (1.65,7.65);
\draw (1.65,7.65) -- (9.65,3.15) -- (64.65,3.15) -- (64.65,2.15) -- (9.65,2.15) -- (1.65,6.65);
\draw (64.65,3.15) -- (65.25,3.75);
\draw (64.65,2.15) -- (65.25,2.75);
\fill[fill=black,fill opacity=0.1] (8.23,-6.77) -- (8.23,3.23) -- (11.77,6.77) --(71.77,6.77) --  (71.77,-3.23) -- (68.23,-6.77) --  cycle;
\draw (8.23,3.23) -- (68.23,3.23) -- (68.23,-6.77) -- (8.23,-6.77) ;
\draw (11.77,6.77) -- (71.77,6.77) -- (71.77,-3.23);
\draw [dash pattern=on 2pt off 2pt] (71.77,-3.23) -- (11.77,-3.23);
\draw [dash pattern=on 2pt off 2pt] (11.77,-3.23)-- (11.77,6.77);
\draw [dash pattern=on 2pt off 2pt] (8.23,-6.77)-- (8.23,3.23);
\draw (71.77,6.77)-- (68.23,3.23);
\draw (71.77,-3.23)-- (68.23,-6.77);
\draw [dash pattern=on 2pt off 2pt] (8.23,-6.77)-- (11.77,-3.23);
\draw [dash pattern=on 2pt off 2pt] (8.23,3.23)-- (11.77,6.77);
\fill[fill=black,fill opacity=0.2] (8.59,2.59) -- (11.41,5.41) -- (66.41,5.41) -- (66.41,-2.59) -- (63.59,-5.41) -- (8.59,-5.41) -- cycle;
\draw (11.41,5.41) -- (66.41,5.41) -- (66.41,-2.59);
\draw [dash pattern=on 2pt off 2pt] (66.41,-2.59) -- (11.41,-2.59);
\draw (8.59,2.59) -- (63.59,2.59) -- (63.59,-5.41) -- (8.59,-5.41);
\draw [dash pattern=on 2pt off 2pt] (11.41,-2.59)-- (11.41,5.41);
\draw [dash pattern=on 2pt off 2pt] (8.59,-5.41)-- (8.59,2.59);
\draw [dash pattern=on 2pt off 2pt] (8.59,2.59)-- (11.41,5.41);
\draw [dash pattern=on 2pt off 2pt] (8.59,-5.41)-- (11.41,-2.59);
\draw (66.41,5.41)-- (63.59,2.59);
\draw (66.41,-2.59)-- (63.59,-5.41);
\end{tikzpicture}
\end{center}
\caption{Tentacles $T_1'$, $T_1$ and $T_2'$}
\label{pic:tentacle}
\end{figure}

Set $T'_1 = T'^S_{1}$, $T_1 = T^S_{1}$.
We need a \textit{shifting} mapping
% HAVE NOT CHECKED IN DETAILS - I BELIEVE THE PICTURES :)

%\begin{equation*}
$$
    (s_k)_{\hat\ve(k)}(t,s) =
    \begin{cases}
        s - (\hat{r}_{k-1} - b_{k-1})(\hat{v}_k)_n,  & \text{if } t \in (\hat{r}_{k-1}, c_k+\hat{r}_{k}],\\
        s - \frac{(\hat{v}_k)_n (\hat{r}_{k-1} - b_{k-1})}{\hat{r}_{k}-\hat{r}_{k-1}}(\hat{r}_{k}-t), & \text{if } t \in (\hat{r}_{k}, \hat{r}_{k-1}],\\
        s, & \text{if } t \in (0,\hat{r}_{k}].
    \end{cases}
$$
%\end{equation*}
For $x\in T'^S_{k}$ define
\begin{equation*}
(S_k)_{\hat\ve(k)}(x_1,\dots,x_n): = \bigl(x_1, \dots, x_{n-1}, (s_k)_{\hat\ve(k)}(x_1, x_n)\bigr).
\end{equation*}
Note that we have shifted the $x_n$ coordinate by
$(\hat{z}_{\ve(k)}-\hat{z}_{\ve(k-1)})_n=\hat{r}_{k-1} (\hat{v}_k)_n$
down (see \eqref{qqq}), i.e.\ we have moved the right part of $k$-tentacle $T'^S_{k}$ to the height of $(k-1)$-tentacle $T^S_{k-1}$
(see Fig.~\ref{pic:two_gen}), and then we moved it by $b_{k-1}(\hat{v}_k)_n$ up so that the position of different $T'^S_{k}$ is different and they are again above each other in the $(k-1)$-tentacle $T^S_{k-1}$ (of height $2b_{k-1}$).

It is easy to see that the Jacobian of this mapping is equal to $1$ and hence it does not change the measure of the tentacles.
We can estimate its derivative as
$$
|D(S_k)_{\hat\ve(k)}| \approx \max\left\{1, \frac{(\hat{v}_k)_n (\hat{r}_{k-1}-b_{k-1})}{\hat{r}_{k} -\hat{r}_{k-1}}\right\}
\approx 1
$$
and moreover
$$|D(S_k)^{-1}_{\hat\ve(k)}| \approx 1.$$

For $\hat\ve(k)=(v_1,\dots,v_k)$ we denote $\hat\ve(j) = (v_1,\dots,v_j)$
and we define
$$
S_{\hat\ve(k)}:=(S_1)_{\hat\ve(1)}\circ \cdots \circ (S_k)_{\hat\ve(k)}.
$$
\begin{remark}\label{rem:sk}
    Note that for each $x\in T'^S_{k} \cap (Q(0,\hat{r}_{j-1})\setminus Q(0,\hat{r}_j))$
    mapping $S_{\hat\ve(k)}(x)$
    is a composition of $k-1$ translations and one bending $(S_j)_{\hat\ve(j)}$
    with
    ${|D(S_j)^{-1}_{\hat\ve(j)}| \approx |D(S_j)_{\hat\ve(j)}| \approx 1}$.
    Hence, this composition is also bilipschitz with a constant that does not depend on $k$.
\end{remark}

Let us define the $k$-th generation as
$$
T'_{\hat\ve(k)}:=S_{\hat\ve(k)}(T'^S_{\hat\ve(k)}) \text{ and } T_{\hat\ve(k)}:=S_{\hat\ve(k)}(T^S_{\hat\ve(k)}), $$
$$
P'_{\hat\ve(k)}:=T'_{\hat\ve(k)}\setminus Q(\hat{z}_{\hat{v}(k)},\hat{r}_k)
\text{ and }P_{\hat\ve(k)}:=T_{\hat\ve(k)}\setminus Q(\hat{z}_{\hat{v}(k)},\hat{r}_k).
$$
This definition ensures \eqref{subset} as
$b_k > 2^n d_{k+1}$ (see Fig.~\ref{pic:two_gen}) by \eqref{choicebd}.
Since the shifting map does not change the volume, we obtain from \eqref{meas1} that %and \eqref{meas2} that
\eqn{meas3}
$$
|P_{\hat\ve(k)} | \approx b_k^{n-1}\text{ and }
|P'_{\hat\ve(k)}|\approx d_k^{n-1}.
%|T'_k \setminus T_k |\approx |T_k\setminus Q(0,\hat{r}_{k})| \approx 2^{-\alpha k(\beta+1) (n-1)}.
$$
% and
% \begin{equation}\label{meas4}
% \Bigl|(T'_k \setminus T_k) \cap \bigl(Q(0,\hat{r}_{j-1})\setminus Q(0,\hat{r}_j)\bigr)\Bigr|\approx
% 2^{ -\alpha k (\beta+1) (n-1) - j \beta}
% \end{equation}
% for
% $1\leq j \leq k-1$.

%We define
%$$
%l_y:=:=\bigcap_{k=1}^{\infty}T_{\hat{\ve}(k)}.
%$$
It is clear that the diameter of $l_x$, which is defined by \eqref{def:l}, is bigger than the diameter of $l^S$ (see \eqref{def:l} and \eqref{def:ls})  and hence $l_x$ is a nontrivial continuum.
Moreover,
%$\mathcal{L}_n \bigl(\bigcup_{x\in C_B^T} l_x\bigr) = 0$
%since
\eqn{ahoj}
$$
\mathcal{L}_n \bigl(\bigcup_{x\in C_B^T} l_x\bigr) = 0,\text{ since }
\mathcal{L}_n \Bigl(\bigcup_{\ve(k)\in\hat{\mathbb{V}}^{k}} T_{\ve(k)}'\Bigr) \leq
2^{nk} \bigl(d_k^{n-1} + \hat{r}_k^n\bigr) \xrightarrow[k\to\infty]{} 0.
$$

%\begin{figure}[h]\label{pic:tentacle_sq}
%\center{\includegraphics[width=1\linewidth]{tentacle_sq.jpg}}
%\caption{Tentacle squeezing.}
%\end{figure}

\begin{figure}[ht]
\begin{center}
\begin{tikzpicture}[line cap=round,line join=round,>=triangle 45,x=0.15cm,y=0.15cm, yscale=0.7]
\clip(12,-25) rectangle (115,25);
\fill[fill=black,fill opacity=0.5] (20,3) -- (94,3) -- (94,-3) -- (20,-3) -- cycle;
\draw (20,-20)-- (20,20)-- (-20,20)-- (-20,-20)-- cycle;
\draw (20,10)-- (100,10)-- (100,-10)-- (20,-10);
\draw (20,3)-- (94,3)-- (94,-3)-- (20,-3);
\draw [dash pattern=on 2pt off 2pt] (94,-23) -- (94,25);
\draw [thick, color=green] (94,-10) -- (94,10);
\draw [dash pattern=on 2pt off 2pt] (20,-23) -- (20,25);
\draw [color=black,dash pattern=on 2pt off 2pt] (40,-23) -- (40,25);
\draw [thick, color=blue] (40,-10) -- (40,10);
\draw [dash pattern=on 2pt off 2pt] (110,-23) -- (110,25);
\fill [fill=black,fill opacity=0.35] (20,3) -- (20,10) -- (40,10) -- (40,3) -- cycle;
\fill [fill=black,fill opacity=0.2] (40,10) -- (40,3) -- (94,3) -- (94,10) -- cycle;
\fill [fill=black,fill opacity=0.35] (20,-3) -- (20,-10) -- (40,-10) -- (40,-3) -- cycle;
\fill [fill=black,fill opacity=0.2] (40,-10) -- (40,-3) -- (94,-3) -- (94,-10) -- cycle;
\fill [fill=black,fill opacity=0.05] (94,10) -- (94,-10) -- (100,-10) -- (100,10) -- cycle;
\draw[color=black] (43,23) node {\scriptsize $\hat{r}_{k-1}$};
\draw[color=black] (96,23) node {\scriptsize $a_k$};
\draw[color=black] (113,23) node {\scriptsize $\hat{r}_0$};
\draw[color=black] (23,23) node {\scriptsize $\hat{r}_k$};
\draw[color=black] (97,0) node {\scriptsize ${T}'^S_k$};
\draw[color=black] (88,0) node {\scriptsize ${T}^S_k$};
\end{tikzpicture}
\begin{tikzpicture}[line cap=round,line join=round,>=triangle 45,x=0.15cm,y=0.15cm, yscale=0.7]
\clip(12,-25) rectangle (115,25);
\fill[fill=black,fill opacity=0.5] (20,3) -- (30,3) -- (30,-3) -- (20,-3) -- cycle;
\draw (20,-20)-- (20,20)-- (-20,20)-- (-20,-20)-- cycle;
\draw [dash pattern=on 2pt off 2pt] (60,-23) -- (60,25);
\draw [dash pattern=on 2pt off 2pt] (110,-23) -- (110,25);
\draw [dash pattern=on 2pt off 2pt] (20,-23) -- (20,25);
\draw [dash pattern=on 2pt off 2pt] (30,-23) -- (30,25);
\draw [color=black,dash pattern=on 2pt off 2pt] (40,-23) -- (40,25);
\draw (30,3)-- (30,-3);
\draw (20,10)-- (60,10)-- (60,-10)-- (20,-10);
\draw (20,3)-- (30,3);
\draw (30,-3)-- (20,-3);

\draw [thick, color=blue] (22.7,-3) to [out=-10,in=170] (30,-4)  to [out=-10,in=100] (40,-10);
\draw [thick, color=blue] (22.7,3) to [out=10,in=-170] (30,4) to [out=10,in=-100] (40,10);
\draw [thick, color=blue] (22.7,3)-- (22.7,-3);

\draw [thick, color=green] (30,-3) to [out=-10,in=170] (48,-5) to [out=-10,in=100] (58,-10);
\draw [thick, color=green] (30,3) to [out=10,in=-170] (48,5) to [out=10,in=-100] (58,10);
\draw [thick, color=green] (30,3)-- (30,-3);

\fill [color=black,fill opacity=0.05] (60,-10) -- (58,-10) to [out=100,in=-10](48,-5)  to [out=170,in=-10](30,-3) -- (30,3) to [out=10,in=-170] (48,5) to [out=10,in=-100] (58,10) -- (60,10)--cycle;
\fill [color=black,fill opacity=0.2] (22.7,-3) to [out=-10,in=170] (30,-4) to [out=-10,in=100] (40,-10) -- (58,-10) to [out=100,in=-10](48,-5) to [out=170,in=-10](30,-3) -- cycle;
\fill [color=black,fill opacity=0.2] (22.7,3) to [out=10,in=-170] (30,4) to [out=10,in=-100] (40,10) -- (58,10) to [out=-100,in=5](48,5) to [out=-170,in=10](30,3) -- cycle;
\fill [color=black,fill opacity=0.35] (22.7,-3) to [out=-10,in=170] (30,-4) to [out=-10,in=100] (40,-10) -- (20,-10) -- (20,-3) -- cycle;
\fill [color=black,fill opacity=0.35] (22.7,3) to [out=10,in=-170] (30,4) to [out=10,in=-100] (40,10) -- (20,10) -- (20,3) -- cycle;
\begin{scriptsize}
\draw[color=black] (43,23) node {\scriptsize $\hat{r}_{k-1}$};
\draw[color=black] (63,23) node {\scriptsize $2\hat{r}_{k-1}$};
\draw[color=black] (113,23) node {\scriptsize $\hat{r}_0$};
\draw[color=black] (23,23) node {\scriptsize $\hat{r}_k$};
\draw[color=black] (34.5,23) node {\scriptsize $2\hat{r}_k=\tilde{a}_k$};
\draw[color=black] (57,0) node {\scriptsize $\tilde{T}'^S_k$};
\draw[color=black] (27,0) node {\scriptsize $\tilde{T}^S_k$};
\end{scriptsize}
\end{tikzpicture}
\end{center}
\caption{Tentacle squeezing $H_k$.}
\label{pic:squeezing}
\end{figure}
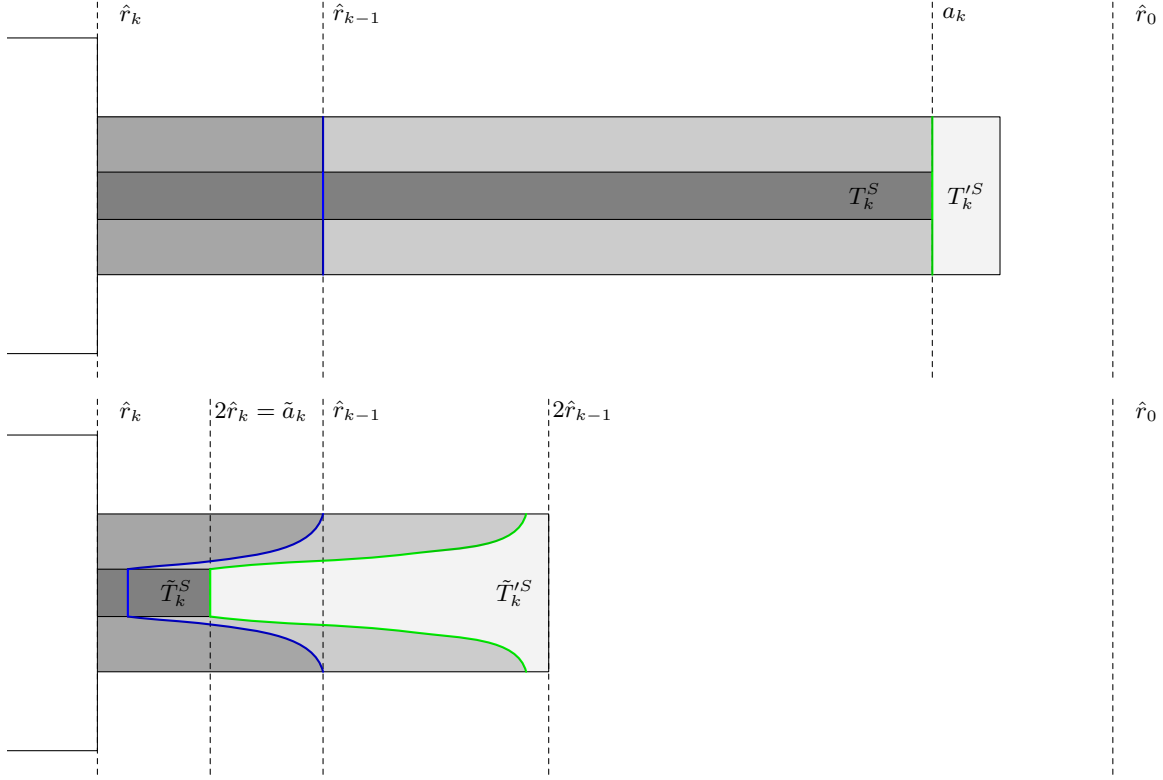

% Define $h_k$ as $h_{k-1}$ on $Q_0 \setminus T'_k$.
\subsection{Squeezing inside tentacles}
The aim of this section is to obtain a mapping which is identity outside the tentacles and squeezes each continuum
$l_x$ onto $x$ for every $x\in C_B^T$.

Analogously to Section~\ref{tentacles} we define parameters which describe the sizes of squeezed tentacles. We set
$$
\begin{aligned}
    \tilde{a}_k & = 2 \hat{r}_{k}\approx 2^{-k(\beta+1)},\\
    \tilde{c}_k & = \tilde{a}_{k-1} = 2 \hat{r}_{k-1} \approx 2^{-k(\beta+1)}.\\
\end{aligned}
$$
With these parameters we consider for each $k\in \en$
$$
\tilde{P}'_k:=P(d_k,\hat{r}_k,\tilde{c}_k)
\quad \text{and} \quad
\tilde{P}_k:=P(b_k,\hat{r}_k,\tilde{a}_k).
$$
Now the `squeezed' tentacles (see Fig.~\ref{pic:squeezing}) are defined by
$$
\begin{aligned}
\tilde{T}'^S_k:=Q(0,\hat{r}_k)\cup \tilde{P}'_k
& \quad \text{and} \quad
\tilde{T}^S_k:=Q(0,\hat{r}_k)\cup \tilde{P}_k, \\
\tilde{T}'^S_{\hat\ve(k)}:=\hat{z}_{\hat\ve(k)}+\tilde{T}'^S_k
& \quad \text{and} \quad
\tilde{T}^S_{\hat\ve(k)}:=\hat{z}_{\hat\ve(k)}+\tilde{T}^S_k. %\textcolor{blue}{\cap T'^S_{\hat\ve(k)}?}
\end{aligned}
$$
%and estimate
%\eqn{meas5}
%$$
%|\tilde{P}_{\hat\ve(k)} | \approx \tilde{a}_k b_k^{n-1}\text{ and }
%|\tilde{P}'_{\hat\ve(k)}|\approx d_k^{n-1}.
%|\tilde{T}^S_{\hat\ve(k)} \setminus Q(0,\hat{r}_k)| \approx 2^{-\alpha k(\beta+1) (n-1)}\cdot 2^{-k(\beta+1)} = %2^{-k(\beta+1)[\alpha  (n-1) +1]}.
%$$

With the help of piecewise linear mapping from Section~\ref{sec:tentacle} we can squeeze the `straight' tentacles. The main idea of this construction is that points have zero capacity in $W^{1,n-1}(\R^{n-1})$, i.e.\ the correct truncation of the function
$\log\log\frac{1}{|x|}$ has small support, value $1$ at $0$ and arbitrarily small norm in $W^{1,n-1}$. For $p<n-1$ it would be enough to work with piecewise affine mappings instead of $\log\log\frac{1}{|x|}$.

\prt{Lemma}
\begin{proclaim}\label{thmH}
Let $n\geq 3$, $\delta_k>0$, $\beta\geq n+1$ and $k\in\en$. Then we can find small enough $d_k>b_k>0$ and a bilipschitz mapping $H^S_k\colon Q(0,1)\to Q(0,1)$ such that $H^S_0(x)=x$ for every $x\in Q(0,1)$
$$
H^S_{k}(x)=H^S_{k-1}(x)\text{ for each }x\notin P_k',\ H_k^S(x)=x\text{ for }x\in Q(0,\hat{r}_k)
$$
$$\text{ and }
H^S_k\text{ maps }P_k\text{ onto }\tilde{P}_k\text{ linearly}.
%H^S_k({T'}_k^S)=\tilde{T'}_k^S,
$$
%\eqn{aha}
%$$
%H^S_k(x)=[l_k(x_1),x_2,\hdots,x_n],\ l_k\text{ is linear on }[\hat{r}_i,\hat{r}_{i-1}],\
%i\in\{2,\hdots,k-1\}\text{ and }|l'_k|\leq 1\text{ there}.
%$$
%For each $k\in\en$ and $1\leq j \leq k$ we estimate its derivative as
Furthermore, $|DH^S_k(x)|\leq 1$ for $x\in P_k$ and
\eqn{defhk}
$$
\begin{aligned}
\int_{P_k'} |DH^S_k(x)|^{n-1}\; dx&\leq \delta_k.
\end{aligned}
$$
\end{proclaim}
\begin{proof}
Set $H^S_0(x)=x$ and proceed by induction. We define
\eqn{ee}
$$
H_k^S(x)=H_{k-1}^S(x)\text{ for }x\notin P_k'
$$
%as identity on the whole $Q(0,\hat{r}_k)$ and further we define
and it remains to define it on $P_k'$.
Since $H_{k-1}^S$ is identity on
$\{x_1\leq \hat{r}_{k-1}\}$ and $P_k'\cap\{x_1\geq \hat{r}_{k-1}\}\subset P_{k-1}$ (see \eqref{choicebd}) where $H_{k-1}^S$  is linear we obtain that on $\partial P_k'$ we have
\eqn{eee}
$$
\begin{aligned}
&H^S_k(x)=[l_{k-1}(x_1),x_2,\hdots,x_n], \text{ where }
l_{k-1}(x_1)=x_1\text{ for }x_1\leq \hat{r}_{k-1}\text{ and }\\
&\text{ for }x_1\in[\hat{r}_{k-1},c_k]\text{ it is linear with }l_{k-1}(\hat{r}_{k-1})=\hat{r}_{k-1}
\text{ and }l_{k-1}(a_{k-1})=\tilde{a}_{k-1}. \\
\end{aligned}
$$
As $\tilde{a}_{k-1}<a_{k-1}$ we know that the derivative $|l'_{k-1}|\leq 1$ there.

Further, we define it for $x\in\{a_k\}\times [-d_k,d_k]^{n-1}$ as
%identity in the last $n-1$ coordinates and in the first coordinate as
$$
\begin{aligned}
H_k^S(x)&=[\phi_k(x),x_2,\hdots,x_n]\text{ where }\\
\phi_k(x)&:=l_{k-1}(a_k)-\Bigl(\log\log \frac{1}{\max\{b_k, |[x_2,\hdots,x_n]|_{\infty}\}}-\log\log\frac{1}{d_k}\Bigr), \\
\end{aligned}
$$
where $|[x_2,\hdots,x_n]|_{\infty}=\max\{|x_2|,\hdots,|x_n|\}$. Then it is easy to see $(H_k^S(x))_1=l_{k-1}(a_k)$ when $x_1=a_k$ and $|[x_2,\hdots,x_n]|_{\infty}=d_k$ and thus it agrees with \eqref{eee} there.
Moreover, we fix $d_k$ small enough in such a way as ($C_{\eqref{fixd}}$ is a constant whose exact value we specify later)
\eqn{fixd}
$$
 \frac{2^{(\beta+1)k(n-1)}}{\log^{n-2}\frac{1}{d_k}}<C_{\eqref{fixd}}\delta_k
$$
and we fix $b_k<d_k$ so that (see Fig. \ref{pic:squeezing})
\eqn{fff}
$$
\text{for } |[x_2,\hdots,x_n]|_{\infty}=b_k \text{ we have }
\phi_k(x)=l_{k-1}(a_k)-\Bigl(\log\log \frac{1}{b_k}-\log\log\frac{1}{d_k}\Bigr)=\tilde{a}_k.
$$
For every $x\in P_k$ we have
$|[x_2,\hdots,x_n]|_{\infty}\leq b_k$ and thus $\phi_k(x)=\tilde{a}_k$. Therefore for every $x\in P_k$ we can define
\eqn{aaa}
$$
H_k^S(x)=\bigl[l_{k}(x_1),x_2,\hdots,x_n\bigr]
\text{ where }l_k\text{ is linear with } l_{k}(\hat{r}_{k})=\hat{r}_{k}\text{ and }l_{k}(a_{k})=\tilde{a}_{k}.
$$
It is easy to see that $|DH_k^S|\leq 1$ there and that this agrees with \eqref{eee} used for $k-1$  before.
Finally on the hyperplane $x\in\{\hat{r}_{k-1}\}\times [-d_k,d_k]^{n-1}$ we define it as (see Fig. \ref{pic:squeezing})
%identity in the last $n-1$ coordinates and in the first coordinate as
$$
\begin{aligned}
H_k^S(x)&=[\psi_k(x),x_2,\hdots,x_n]\text{ where }\\
\psi_k(x)&:=l_{k-1}(\hat{r}_{k-1})-A_k\Bigl(\log\log \frac{1}{\max\{b_k, |[x_2,\hdots,x_n]|_{\infty}\}}-\log\log\frac{1}{d_k}\Bigr). \\
\end{aligned}
$$
As before it agrees with \eqref{eee} for $x_1=\hat{r}_{k-1}$ and $|[x_2,\hdots,x_n]|_{\infty}=d_k$. The constant
$A_k$ is chosen so that for $x\in P_k\cap\{x_1=\hat{r}_{k-1}\}$, i.e.\ for $|[x_2,\hdots,x_n]|_{\infty}\leq b_k$, it goes along with \eqref{aaa}. By this and \eqref{fff} we obtain
$$
\hat{r}_k\leq l_{k-1}(\hat{r}_{k-1})-A_k\Bigl(\log\log\frac{1}{b_k}-\log\log\frac{1}{d_k}\Bigr)
=l_{k-1}(\hat{r}_{k-1})+A_k\Bigl(\tilde{a}_k-l_{k-1}(a_k)\Bigr)
$$
and hence
$$
A_k\leq \frac{l_{k-1}(\hat{r}_{k-1})-\hat{r}_k}{l_{k-1}(a_k)-\tilde{a}_k}
\leq \frac{\hat{r}_{k-1}-\hat{r}_k}{\hat{r}_{k-1}-2\hat{r}_k}\leq C
$$
and so $A_k$ is bounded by a constant independent of $k$.

For every $[x_2,\hdots,x_n]\in [-d_k,d_k]^{n-1}$ we use linear interpolation between values on four hyperplanes ($x_1=\hat{r}_k$, $x_1=\hat{r}_{k-1}$, $x_1=a_k$ and $x_1=c_k$) with the help of the function $h$ from Section~\ref{sec:tentacle} and we define
$$
H^S_k(x)=\Bigl[h\bigl(x_1;[\hat{r}_{k},\hat{r}_{k}],[\hat{r}_{k-1},\psi_k(x)],[a_k,\phi_k(x)],[c_k,l_{k-1}(c_k)]\bigr),x_2,\hdots,x_n\Bigr]
\text{ for }x\in P_k'.
$$
By \eqref{ee} and \eqref{eee} this mapping is continuous.
The mapping $H_k^S$ is bilipschitz on all parts (whilst the bilipschitz constant depends on $k$) and hence it follows immediately that it is bilipschitz on $Q(0,1)$.

It remains to estimate the integrability of the derivative. By \eqref{estder} we obtain that the derivative with respect to the first coordinate can be estimated as
$$
|D_1H_k^S(x)|\leq \begin{cases}
    \frac{\psi_k(x)-\hat{r}_k}{\hat{r}_{k-1}-\hat{r}_k},  & \text{ for } \hat{r}_k< x_1 < \hat{r}_{k-1}, \\
    \frac{\phi_k(x)-\psi_k(x)}{a_k-\hat{r}_{k-1}},  & \text{ for } \hat{r}_{k-1}< x_1 < a_k, \\
    \frac{l_{k-1}(c_k)-\phi_k(x)}{c_k-a_k}, & \text{ if } a_k < x_1 < c_k.
\end{cases}
$$
Since $\phi_k(x)$ takes values between $l_{k-1}(a_k)$ and $\tilde{a}_k$ (see \eqref{fff}) and
$\psi_k(x)$ takes values between $l_{k-1}(\hat{r}_{k-1})=\hat{r}_{k-1}$ and $\tilde{a}_k$ we can estimate this by the universal constant $C$ (where $C$ does not depend on $k$).
Furthermore, by \eqref{defh} we know that we can estimate the derivative with respect to other coordinates by the constant multiple of the corresponding derivative of
$$
\frac{\psi_k(x)-\hat{r}_k}{\hat{r}_{k-1}-\hat{r}_k}+\frac{\phi_k(x)-\psi_k(x)}{a_k-\hat{r}_{k-1}}+\frac{l_{k-1}(c_k)-\phi_k(x)}{c_k-a_k}.
$$
Since $A_k\leq C$ we can estimate this by
$$
C \max\Bigl\{\frac{1}{\hat{r}_{k-1}-\hat{r}_k},\frac{1}{a_k-\hat{r}_{k-1}},\frac{1}{c_k-a_k}\Bigr\}
\Bigl|D\Bigl(\log\log \frac{1}{\max\{b_k, |[x_2,\hdots,x_n]|_{\infty}\}}\Bigr)\Bigr|
$$
The maximum of the three terms can be estimated by $C \frac{1}{\hat{r}_k}\leq C 2^{(\beta+1) k}$ and thus
we can estimate the derivative with respect to other coordinates as
$$
|D_jH_k^S(x)|\leq
\begin{cases}
\frac{C 2^{(\beta +1)k}}{|[x_2,\hdots,x_n]|_{\infty}\log\frac{1}{|[x_2,\hdots,x_n]|_{\infty}}}&\text{ for }b_k<|[x_2,\hdots,x_n]|_{\infty}<d_k,\\
0&\text{ for }|[x_2,\hdots,x_n]|_{\infty}<b_k.\\
\end{cases}
$$
Now a simple change to polar/spherical coordinates in $\R^{n-1}$
and \eqref{fixd} gives us
$$
\begin{aligned}
\int_{P_k'} |DH_k^S(x)|^{n-1}\; dx&\leq C 2^{(\beta+1)k(n-1)}
\int_{_{P_k'}} \frac{1}{|[x_2,\hdots,x_n]|^{n-1}_{\infty}\log^{n-1}\frac{1}{|[x_2,\hdots,x_n]|_{\infty}}}\; dx\\
&\leq C 2^{(\beta+1)k(n-1)}\int_0^{d_k}\frac{1}{r^{n-1}\log^{n-1}\frac{1}{r}}r^{n-2}\; dr\\
&\leq C 2^{(\beta+1)k(n-1)}\frac{1}{\log^{n-2}\frac{1}{d_k}}<C C_{\eqref{fixd}} \delta_k<\delta_k,
\\
\end{aligned}
$$
where we have chosen $C_{\eqref{fixd}}$ in \eqref{fixd} so that the last inequality holds.
\end{proof}
% We define $H_0(x)=x$ for every $x$ and further by induction
% $$
% H_k(x):=
% \begin{cases}
% H_{k-1}(x)&\text{ for }x\in Q_0 \setminus \bigcup_{\ve(k)\in\hat{\mathbb{V}}^{k}} T'_{\hat\ve(k)},\\
% S_{\hat\ve(k)}  \circ (H^S_k +\hat{z}_{\hat\ve(k)}) \circ (S^{-1}_{\hat\ve(k)} (x)-\hat{z}_{\hat\ve(k)}) &\text{ for }x\in T'_{\hat\ve(k)}. \\
% \end{cases}
% $$
% Since $S_{\hat\ve(k)} $ is a bilipschitz mapping (whose constant does not depend on $k$) we  can estimate the derivatives of $h_k$ (and thus also of $h$) similarly to Lemma~\ref{thmH}.

Above we have defined `straight' tentacles $T^{'S}_{\hat{\ve}(k)}$ and we have squeezed them by $H^S_k$ onto squeezed `straight' tentacles $\tilde{T}^{'S}_{\hat{\ve}(k)}$. Analogously we take `real' (=twisted) tentacles $T_{\hat{\ve}(k)}'$ and we squeeze inside them by $H_k$ to obtain `real' squeezed tentacles $\tilde{T}'_{\hat\ve(k)}$.
For $k\geq 1$ we define
\eqn{defHk}
$$
H_k(x):=
S_{\hat\ve(k)}  \circ (H^S_k +\hat{z}_{\hat\ve(k)}) \circ (S^{-1}_{\hat\ve(k)} (x)-\hat{z}_{\hat\ve(k)})
\quad \text{ for }\quad x\in T_{\hat{\ve}(k)}',
$$
$$
\tilde{T}'_{\hat\ve(k)}:=H_k(T'_{\hat\ve(k)})
 \quad \text{and} \quad
\tilde{T}_{\hat\ve(k)}:=H_k(T_{\hat\ve(k)}).
$$

%\textcolor{blue}
%{Note that
%\begin{multline}\label{meas6}
%\Bigl|H_k\bigl(T_{\hat{v}(k)}\cap\bigl(Q(z_{\hat{v}(k)},\hat{r}_{j-1})\setminus %Q(z_{\hat{v}(k)},\hat{r}_{j})\bigr)\bigr)\Bigr|\approx
%2^{-\alpha k(\beta+1) (n-1)}\cdot 2^{-(k+j)(\beta+1)}\\
%\leq
%2^{-k(\beta+1)\alpha  (n-1)}
%\end{multline}
%and
%\begin{multline}\label{meas7}
%\Bigl|H_k(\bigl(T'_{\hat{v}(k)} \setminus T_{\hat{v}(k)}) \cap \bigl(Q(z_{\hat{v}(k)},\hat{r}_{j-1})\setminus %Q(z_{\hat{v}(k)},\hat{r}_{j})\bigr)\bigr)\Bigr|\lesssim
%2^{-\alpha k(\beta+1) (n-1)}\cdot 2^{-(k+j)(\beta+1)}\\
%\leq
%2^{-k(\beta+1)\alpha  (n-1) }
%\end{multline}
%for
%$1\leq j \leq k-1$.
%}

\prt{Theorem}
\begin{proclaim}\label{thmH2}
Let $n\geq 3$, $\tilde{\delta}_k>0$, $\beta\geq n+1$ and $k\in\en$. Then we can find small enough $d_k>b_k>0$ and a bilipschitz mapping
 $h_k\colon Q(0,1)\to Q(0,1)$ such that $h_0(x)=x$ for every $x\in Q(0,1)$,
 \eqn{www}
$$
h_{k-1}(x)=h_k(x)\text{ for }
x\notin \bigcup_{\hat{\ve}(k)\in\hat{\mathbb{V}}^{k}} P_{\hat{\ve}(k)}',\ h_k(x)=x\text{ for }x\in \hat{Q}_{\hat{\ve}(k)}
\text{ and } h_k(P_{\hat{\ve}(k)})=\tilde{P}_{\hat{\ve}(k)}.
%\ h_k({T}'_{\ve(k)})=\tilde{T}'_{\ve(k)}.
$$
% $$
% \text{ and } J_{h_k}(x) > 0 \text{ for all } x \in Q(0,1).
% $$
We can estimate the integral of its derivative as
\eqn{defsmallhk}
$$
\int_{\bigcup_{\hat{v}(k)\in \hat{\mathbb{V}}^k} P^{'}_{\hat{\ve}(k)}} |Dh_k(x)|^{n-1}\; dx\leq \tilde{\delta}_k.
$$

Moreover, a pointwise limit $h$ of $h_k$ is continuous,
$J_{h}(x) > 0$ a.e., and
$$
h(l_x)=x\text{ for every }x\in C_B^T,
$$
where
$l_x$ is defined by \eqref{def:l}.
\end{proclaim}
\begin{proof}
We set $h_0(x)=x$ and further we define (see \eqref{defHk})
\eqn{defrealhk}
$$
h_k(x)=
\begin{cases}
h_{k-1}(x)&\text{ for }x\notin \bigcup_{\hat{\ve}(k)\in\hat{\mathbb{V}}^{k}} T_{\hat{\ve}(k)}',\\
H_k(x)&\text{ for }x\in T_{\hat{\ve}(k)}',\\
\end{cases}
$$
which clearly fulfills \eqref{www} since $H_k(x)=h_{k-1}(x)=x$ on all $Q(\hat{z}_{\hat{\ve}(k)}, \hat{r}_k)$.

We have $2^{nk}$ different sets $T^{'}_{\hat{\ve}(k)}$ and all of them are bilipschitz copy of $T_k'^S$,
mappings $S_{\hat\ve(k)}$ are bilipschitz with a constant that does not depend on $k$, and hence we obtain by Lemma \ref{thmH} that
$$
\int_{\bigcup_{\hat{v}(k)\in \hat{\mathbb{V}}^k} T^{'}_{\hat{\ve}(k)}} |Dh_k(x)|^{n-1}\; dx\leq 2^{nk} C
\int_{T_k'^S} |DH^S_k(x)|^{n-1}\; dx\leq 2^{nk}C\delta_k.
$$
Given $\tilde{\delta}_k$ we set $\delta_k=\frac{1}{C}2^{-nk}\tilde{\delta}_k$ and find $b_k$ and $d_k$ small enough so that \eqref{defhk} and thus also \eqref{defsmallhk} holds.

Outside of $\bigcup_{\ve(k)\in\hat{\mathbb{V}}^{k}} T_{\hat{\ve}(k)}'$ all mapping $h_l$, $l\geq k$, are equal to $h_{k-1}$ and they are therefore bilipschitz there and $J_{h_l}>0$ a.e. It follows that we can define $h=\lim_{k\to\infty} h_k$ and it is defined everywhere outside of (see \eqref{def:l} and \eqref{subset})
$$
\bigcap_{k=1}^{\infty}\bigcup_{\hat{\ve}(k)\in\hat{\mathbb{V}}^{k}} T_{\hat{\ve}(k)}'=\bigcup_{x\in C_B^T} l_x.
$$
Moreover, it is continuous and $J_h>0$ a.e.\ there.
By \eqref{ahoj} we know that
$
\mathcal{L}_n (\bigcup_{x\in C_B^T} l_x )= 0
$
and then $h$ is defined a.e. Since
$$h_k(T_{\hat{\ve}(k)})=\tilde{T}_{\hat{\ve}(k)}\text{ and }\operatorname{diam} \tilde{T}_{\hat{\ve}(k)}\to 0
$$
it is not difficult to see that $h(l_x)=x$ for every
$x\in C_B^T$. The continuity of $h$ everywhere follows.
\end{proof}

\subsection{Counterexample in Theorem \ref{firstthm}}

\begin{proof}[Construction of the counterexample in Theorem \ref{firstthm}]
Define a Cantor-type set $C_A$ of positive measure by
$$
\alpha_k = \frac{1}{2}\left(1+2^{-k\beta}\right).
$$
We need the sequence of functions $g_k$, built in Section~\ref{sec:map_CS}, to map $C_A$ onto the a Cantor-type set $C_B$ with small enough `windows' defined by \eqref{defbeta}, i.e.
$$
\beta_k = 2^{-k \beta}\text{ with }\beta\geq n+1.
$$
According to \eqref{eq:Dg2} in the $i$-th frame $Q'_{\ve(i)}\setminus Q_{\ve(i)}$, $i\leq k$, we have
%\begin{equation}\label{dg}
%    |Dg_k^{AB}(x)| \approx
%    \max \left\{\frac{b_i}{a_i}, \frac{b_{i-1} - b_{i}}{a_{i-1} - a_{i}}\right\} \approx
%    \max \left\{2^{-i\beta}, 1\right\} = 1,
%\end{equation}
%\begin{equation}\label{jac}
%    |J_{g^{AB}_k}(x)| \approx 2^{-i\beta(n-1)}
%\end{equation}
%and
\begin{equation}\label{dg-1}
    |D(g_k)^{-1}(x)| \approx
    \max \left\{\frac{\alpha_i}{\beta_i}, \frac{\alpha_{i-1} - \alpha_{i}}{\beta_{i-1} - \beta_{i}}\right\}
    \approx 2^{i\beta}
\end{equation}
and on $\tilde{Q}_{\ve(k)}$ we have
\begin{equation}\label{dg-2}
    |D(g_k)^{-1}(x)| \approx
    \frac{\alpha_k}{\beta_k}\approx 2^{k\beta}.
\end{equation}

We also need the bilipschitz mapping $L$, defined in Section~\ref{map_CT}, to map $C_B$ to a Cantor tower $C^T_B$,
% $$c_k = 2^{-k \beta},$$
and we have
\begin{equation}\label{dl}
|DL(x)| \leq l, \quad |DL^{-1}(x)| \leq l.
\end{equation}

Let us start from the Cantor tower $C^T_B$ and
apply our mapping $h_k$ from Theorem \ref{thmH2} to squeeze the inner part of the cube.
Then we need a mapping $L^{-1}$ to go from $C^T_B$ to $C_B$,
and $(g_k)^{-1}$ to go to the Cantor set of positive measure $C_A$.
The final mapping $f$ is a pointwise limit of
$$
f_k(x) = (g_k)^{-1} \circ L^{-1} \circ h_k(x)
$$
almost everywhere (see Fig.~\ref{pic:the_mapping}).
Mappings $f_k$ are clearly bilipschitz and below
we show that $f_k\to f$ strongly in $W^{1,n-1}$ and hence $f$ is a strong limit of Sobolev homeomorphisms $f_k$ such that $f_k(x)=x$ on $\partial[-1,1]^n$.
We know that
$g^{-1}=\lim_{k\to \infty} (g_k)^{-1}$ is a homeomorphism which maps $C_B$ onto $C_A$ and that
$L^{-1}$ is a homeomorphism which maps $C_B^T$ onto $C_B$. By Theorem \ref{thmH2} we know that $h=\lim_{k\to\infty}h_k$ is continuous and the standard computation shows that
$$
f(x)=g^{-1}\circ L^{-1}\circ h(x)\text{ and it is a continuous mapping which maps }C_B^T\text{ onto }C_A.
$$
By Theorem \ref{thmH2} we also know that for every $x\in C_B^T$ we have $h(l_x)=x$ and clearly $g^{-1}\circ L^{-1}(x)=y$ where $y$ is the corresponding point in $C_A$ (see \eqref{corespond}). It follows that $f^{-1}(y)$ is a continuum $l_x$ for every
$y\in C_A$. Finally, $J_{L^{-1}}>0$ a.e., $J_h>0$ a.e.\ by Theorem \ref{thmH2}, and by the construction we also have $J_{g^{-1}}>0$ a.e.\ as it is locally equal to some bilipschitz mapping $g_k^{-1}$ on $[-1,1]^n\setminus C_B$.
It is not difficult to see that $f$ is locally bilipschitz on
$[-1,1]^n\setminus \bigcup_{x\in C_B^T}l_x$ and hence we can use the composition formula for derivatives to obtain (see \eqref{ahoj})
$$
J_f(x)=J_{g^{-1}}\bigl(L^{-1}(h(x))\bigr) J_{L^{-1}}\bigl(h(x)\bigr) J_h(x)>0\text{ for a.e. }x\in[-1,1]^n.
$$

\begin{figure}[h]
\begin{center}
\begin{tikzpicture}[line cap=round,line join=round,>=triangle 45,x=0.36cm,y=0.36cm]
\clip(-26,8) rectangle (20,18);

\draw (-25,17)-- (-25,9)-- (-17,9)-- (-17,17) -- cycle;
\draw (-21.3,15.7)-- (-20.7,15.7)-- (-20.7,16.3) -- (-21.3,16.3)-- cycle;
\draw (-20.7,15.9)-- (-17.7,15.9)-- (-17.7,16.1)-- (-20.7,16.1) -- cycle;
\fill[fill=black,fill opacity=0.1] (-20.7,15.9)-- (-17.7,15.9)-- (-17.7,16.1)-- (-20.7,16.1) -- cycle;
\draw (-20.7,14.3) -- (-21.3,14.3) -- (-21.3,13.7) -- (-20.7,13.7) -- cycle;
\draw (-20.7,14.1) -- (-17.7,14.1) -- (-17.7,13.9) -- (-20.7,13.9) -- cycle;
\fill[fill=black,fill opacity=0.1] (-20.7,14.1) -- (-17.7,14.1) -- (-17.7,13.9) -- (-20.7,13.9) -- cycle;
\draw (-21.3,11.7) -- (-20.7,11.7) -- (-20.7,12.3) -- (-21.3,12.3) -- cycle;
\draw (-20.7,11.9) -- (-17.7,11.9) -- (-17.7,12.1) -- (-20.7,12.1) -- cycle;
\fill[fill=black,fill opacity=0.1] (-20.7,11.9) -- (-17.7,11.9) -- (-17.7,12.1) -- (-20.7,12.1) -- cycle;
\draw (-20.7,10.3) -- (-21.3,10.3) -- (-21.3,9.7) -- (-20.7,9.7) -- cycle;
\draw (-20.7,10.1) -- (-17.7,10.1) -- (-17.7,9.9) -- (-20.7,9.9) -- cycle;
\fill[fill=black,fill opacity=0.1] (-20.7,10.1) -- (-17.7,10.1) -- (-17.7,9.9) -- (-20.7,9.9) -- cycle;

% \draw (-1,-3) -- (-1,5) -- (7,5) -- (7,-3) -- cycle;
\draw (-13,17)-- (-13,9)-- (-5,9)-- (-5,17) -- cycle;
\draw (-8.7,15.7) -- (-9.3,15.7) -- (-9.3,16.3) -- (-8.7,16.3) -- cycle;
\draw (-8.7,15.9) -- (-8.3,15.9) -- (-8.3,16.1) -- (-8.7,16.1) -- cycle;
\fill[fill=black,fill opacity=0.1] (-8.7,15.9) -- (-8.3,15.9) -- (-8.3,16.1) -- (-8.7,16.1) -- cycle;
\draw (-9.3,14.3) -- (-8.7,14.3) -- (-8.7,13.7) -- (-9.3,13.7) -- cycle;
\draw (-8.7,14.1) -- (-8.3,14.1) -- (-8.3,13.9) -- (-8.7,13.9) -- cycle;
\fill[fill=black,fill opacity=0.1] (-8.7,14.1) -- (-8.3,14.1) -- (-8.3,13.9) -- (-8.7,13.9) -- cycle;
\draw (-8.7,11.7) -- (-9.3,11.7) -- (-9.3,12.3) -- (-8.7,12.3) -- cycle;
\draw (-8.7,11.9) -- (-8.3,11.9) -- (-8.3,12.1) -- (-8.7,12.1) -- cycle;
\fill[fill=black,fill opacity=0.1] (-8.7,11.9) -- (-8.3,11.9) -- (-8.3,12.1) -- (-8.7,12.1) -- cycle;
\draw (-9.3,10.3) -- (-8.7,10.3) -- (-8.7,9.7) -- (-9.3,9.7) -- cycle;
\draw (-8.7,10.1) -- (-8.3,10.1) -- (-8.3,9.9) -- (-8.7,9.9) -- cycle;
\fill[fill=black,fill opacity=0.1]  (-8.7,10.1) -- (-8.3,10.1) -- (-8.3,9.9) -- (-8.7,9.9) -- cycle;

% \draw (-13,-3) -- (-13,5) -- (-5,5) -- (-5,-3) -- cycle;
\draw (-1,17)-- (-1,9)-- (7,9)-- (7,17) -- cycle;
\draw (0.7,15.3) -- (0.7,14.7) -- (1.3,14.7) -- (1.3,15.3) -- cycle;
\draw (1.3,14.9) -- (4.2,13.9) -- (4.2,14.1) -- (1.3,15.1) -- cycle;
\fill[fill=black,fill opacity=0.1] (1.3,14.9) -- (4.2,13.9) -- (4.2,14.1) -- (1.3,15.1) -- cycle;
\draw (5.3,15.3) -- (5.3,14.7) -- (4.7,14.7) -- (4.7,15.3) -- cycle;
\draw (5.3,15.1) -- (5.4,15.3) -- (4.5,16.1) -- (4.2,16.3) -- (4.7,16.6) -- (4.7,16.5) -- (4.5,16.3) -- (5.6,15.3) -- (5.6,15.2) -- (5.3,14.9) -- cycle;
\fill[fill=black,fill opacity=0.1] (5.3,15.1) -- (5.4,15.3) -- (4.5,16.1) -- (4.2,16.3) -- (4.7,16.6) -- (4.7,16.5) -- (4.5,16.3) -- (5.6,15.3) -- (5.6,15.2) -- (5.3,14.9) -- cycle;
\draw (5.3,10.7) -- (5.3,11.3) -- (4.7,11.3) -- (4.7,10.7) -- cycle;
\draw (5.3,10.9) -- (5.4,10.7) -- (4.5,9.9) -- (4.2,9.7) -- (4.7,9.4) -- (4.7,9.5) -- (4.5,9.7) -- (5.6,10.7) -- (5.6,10.8) -- (5.3,11.1) -- cycle;
\fill[fill=black,fill opacity=0.1] (5.3,10.9) -- (5.4,10.7) -- (4.5,9.9) -- (4.2,9.7) -- (4.7,9.4) -- (4.7,9.5) -- (4.5,9.7) -- (5.6,10.7) -- (5.6,10.8) -- (5.3,11.1) -- cycle;
\draw (0.7,10.7) -- (0.7,11.3) -- (1.3,11.3) -- (1.3,10.7) -- cycle;
\draw (1.3,11.1) -- (4.2,12.1) -- (4.2,11.9) -- (1.3,10.9) -- cycle;
\fill[fill=black,fill opacity=0.1] (1.3,11.1) -- (4.2,12.1) -- (4.2,11.9) -- (1.3,10.9) -- cycle;

% \draw (-25,-3) -- (-25,5) -- (-17,5) -- (-17,-3) -- cycle;
\draw (11,9) -- (11,17) -- (19,17) -- (19,9) -- cycle;
\draw (11.5,13.5) -- (11.5,16.5) -- (14.5,16.5) -- (14.5,13.5) -- cycle;
\draw (14.5,15.5) -- (15.3,13.3) -- (15.3,13.1) -- (14.5,14.5) -- cycle;
\fill[fill=black,fill opacity=0.1] (14.5,15.5) -- (15.3,13.3) -- (15.3,13.1) -- (14.5,14.5) -- cycle;
\draw (15.5,13.5) -- (15.5,16.5) -- (18.5,16.5) -- (18.5,13.5) -- cycle;
\draw (18.5,14.5) -- (18.8,15) -- (18.8,16.7) -- (17.2,16.7) -- (17.3,16.8) -- (17.1,16.8) -- (16.8,16.6) -- (18.6,16.6) -- (18.6,15.8) -- (18.5,15.5) -- cycle;
\fill[fill=black,fill opacity=0.1] (18.5,14.5) -- (18.8,15) -- (18.8,16.7) -- (17.2,16.7) -- (17.3,16.8) -- (17.1,16.8) -- (16.8,16.6) -- (18.6,16.6) -- (18.6,15.8) -- (18.5,15.5) -- cycle;
\draw (15.5,9.5) -- (15.5,12.5) -- (18.5,12.5) -- (18.5,9.5) -- cycle;
\draw (18.5,11.5) -- (18.8,11) -- (18.8,9.3) -- (17.2,9.3) -- (17.3,9.2) -- (17.1,9.2) -- (16.8,9.4) -- (18.6,9.4) -- (18.6,10.6) -- (18.5,10.5) -- cycle;
\fill[fill=black,fill opacity=0.1] (18.5,11.5) -- (18.8,11) -- (18.8,9.3) -- (17.2,9.3) -- (17.3,9.2) -- (17.1,9.2) -- (16.8,9.4) -- (18.6,9.4) -- (18.6,10.6) -- (18.5,10.5) -- cycle;
\draw (11.5,9.5) -- (11.5,12.5) -- (14.5,12.5) -- (14.5,9.5) -- cycle;
\draw (14.5,10.5) -- (15.3,12.7) -- (15.3,12.9) -- (14.5,11.5) -- cycle;
\fill[fill=black,fill opacity=0.1] (14.5,10.5) -- (15.3,12.7) -- (15.3,12.9) -- (14.5,11.5) -- cycle;

\draw [->] (-16.3,13) -- (-13.7,13);
\draw [->] (-4.3,13) -- (-1.7,13);
\draw [->] (7.7,13) -- (10.3,13);
% \draw [->] (-1.7,1) -- (-4.3,1);
% \draw [->] (3,8.5) -- (3,5.7);
% \draw [->] (-13.7,1) -- (-16.3,1);
\begin{scriptsize}
\draw[color=black] (-15,13.7) node {$h_k$};
\draw[color=black] (-3,13.7) node {$L^{-1}$};
\draw[color=black] (9,13.7) node {$g_k^{-1}$};
% \draw[color=black] (3.6,7.2) node {$h_k$};
% \draw[color=black] (-2.7,1.7) node {$L_k^{-1}$};
% \draw[color=black] (-14.7,1.7) node {$g_k^{-1}$};
\end{scriptsize}
\end{tikzpicture}
\end{center}
\caption{Mapping $f_k$.}
\label{pic:the_mapping}
\end{figure}
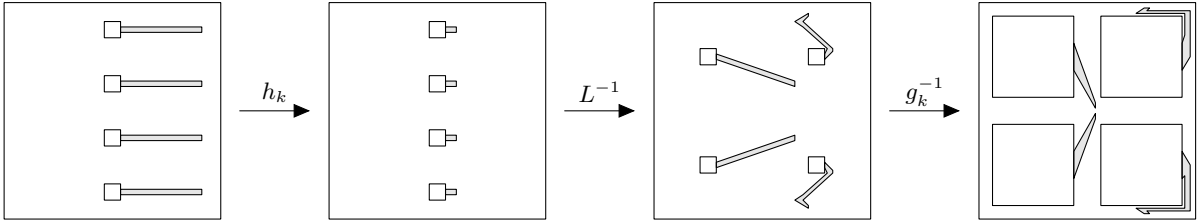

It remains to show that $f\in W^{1,n-1}$. We show that mappings $f_k$ form a Cauchy sequence in
$W^{1,n-1}$. Since $f_k\to f$ pointwise, it is easy to see that $f_k$ converges strongly to $f$.
We have fixed $\beta>n+1$ so that Theorem \ref{Bilip} holds and we set
\eqn{defdeltak}
$$
\tilde{\delta}_k=\frac{2^{-k\beta (n-1)}}{k^2}.
$$
Given this $\tilde{\delta}_k$ we find $d_k>b_k>0$ according to Theorem \ref{thmH2}. Note that all conclusions above ($f$ is continuous, $J_f>0$ a.e.) are valid, but we need this choice of $d_k>b_k$ to show that $f\in W^{1,n-1}$.

%For simplicity we write $g_k=g_k^{AB}$.
%Note that $f_k$ is bilipschitz (as a composition of bilipschitz mappings) and hence we can compute its derivative a.e.\ by the composition of derivatives. With the help of \eqref{dl} we get
%\eqn{chain}
%$$
%|Df_k(x)|\leq |Dg_k^{-1}|\cdot |DL^{-1}|\cdot |Dh_k|
%\leq l \bigl|Dg_k^{-1}(L^{-1}\circ h_k(x))\bigr|\cdot \bigl|Dh_k(x)\bigr|.
%$$

By Theorem \ref{thmH2} we know that $h_{k-1}(x)=h_k(x)$ for every $x\notin \bigcup_{\hat{\ve}(k)\in\hat{\mathbb{V}}^{k}} P_{\hat{\ve}(k)}'$ and clearly by \eqref{star}
$g^{-1}_{k-1}(y)=g^{-1}_{k}(y)$ for $y\notin \bigcup_{\ve(k)\in\mathbb{V}^{k-1}} \tilde{Q}_{\ve(k-1)}$.
In view of  \eqref{goodmap} it follows that
$$
f_k(x)=f_{k-1}(x)\text{ for }x\notin \bigcup_{\hat{\ve}(k)\in\hat{\mathbb{V}}^{k}} P_{\hat{\ve}(k)}'\cup
\bigcup_{\hat{\ve}(k-1)\in\hat{\mathbb{V}}^{k-1}}\hat{Q}_{\hat{\ve}(k-1)}
=:M_k.
$$
Therefore
\eqn{prvni}
$$
\int_{Q(0,1)}|Df_{k}-Df_{k-1}|^{n-1}=\int_{M_k}|Df_{k}-Df_{k-1}|^{n-1}
\leq C \int_{M_k}|Df_{k}|^{n-1}+
C\int_{M_{k}}|Df_{k-1}|^{n-1}.
$$

Note that $f_k$ is bilipschitz (as a composition of bilipschitz mappings) and hence we can compute its derivative a.e.\ by the composition of derivatives. With the help of \eqref{dl} we get
$$
|Df_k(x)|\leq |Dg_k^{-1}|\cdot |DL^{-1}|\cdot |Dh_k|
\leq l \bigl|Dg_k^{-1}(L^{-1}\circ h_k(x))\bigr|\cdot \bigl|Dh_k(x)\bigr|.
$$
By \eqref{dg-1} and \eqref{dg-2} we know that everywhere in $Q(0,1)$ we have $|Dg^{-1}_k|\leq C 2^{k\beta}$ and
$|Dg^{-1}_{k-1}|\leq C 2^{k\beta}$. It follows that
$$
|Df_k(x)|\leq C 2^{k\beta}|Dh_k(x)|\text{ and }
|Df_{k-1}(x)|\leq C 2^{k\beta}|Dh_{k-1}(x)|.
$$

For $x\in\hat{Q}_{\hat{\ve}(k-1)}\setminus \bigcup_{\hat{\ve}(k)\in\hat{\mathbb{V}}^{k}} P_{\hat{\ve}(k)}'$ we know that $h_k(x)=h_{k-1}(x)=x$ by Theorem \ref{thmH2} and hence
$$
\begin{aligned}
\int_{\bigcup_{\hat{\ve}(k-1)\in\mV^{k-1}}\hat{Q}_{\hat{\ve}(k-1)}\setminus \bigcup_{\hat{\ve}(k)\in\hat{\mathbb{V}}^{k}} P_{\hat{\ve}(k)}'}|Df_k|^{n-1}
&\leq
C 2^{k\beta(n-1)}\mathcal{L}^n\Bigl(\bigcup_{\hat{\ve}(k)\in\mV^k}\hat{Q}_{\hat{\ve}(k)}\Bigr)\\
&\leq C 2^{k\beta(n-1)} 2^{nk} (2^{-k} 2^{-k\beta})^n\leq C 2^{-k\beta}. \\
\end{aligned}
$$
With the help of Theorem \ref{thmH2} and \eqref{defdeltak} we obtain
$$
\int_{\bigcup_{\hat{\ve}(k)\in\mV^k}P_{\hat{\ve}(k)}'}|Df_{k}|^{n-1}\leq C 2^{k\beta(n-1)}
\int_{\bigcup_{\hat{\ve}(k)\in\mV^k}P_{\hat{\ve}(k)}'   }|Dh_{k}|^{n-1}
\leq C 2^{k\beta(n-1)}\tilde{\delta}_k\leq \frac{C}{k^2}.
$$
Analogous estimate holds also for $Df_{k-1}$ and hence \eqref{prvni} implies that
$$
\int_{Q(0,1)}|Df_{k}-Df_{k-1}|^{n-1}\leq C 2^{-k\beta}+\frac{C}{k^2}.
$$
Since $2^{-k\beta}+1/k^2$ is a convergent series it follows immediately that $f_k$ form a Cauchy sequence in $W^{1,n-1}$.
It follows that $f\in W^{1,n-1}$.
\end{proof}

%STUPID QUESTION: IS $J_f\in L^1$? IT IS TRUE FOR ALL HOMEOMORPHISMS IN $W^{1,1}$. COULD IT FAIL FOR THEIR LIMIT?
%By \cite{HK}[Thm A.35] we know that
%$$
%%\int_{\Omega}|J_f(x)|\; dx=\int_{\R^n} N(f,\Omega,y)\; dy
%$$
%once $f$ satisfies the $(N)$ condition (it looks that is does - am I kidding?). The RHS=$\infty$ and hence $J_f$ should not be integrable.

%============================================================================================
\section{Injectivity in the domain: counterexample in Theorem \ref{secondthm}}

\subsection{Stretching inside tentacles}

The following gives us an analogy of Lemma \ref{thmH}. We can view the mapping $\tilde{H}_k^S$ as the inverse of $H_k^S$ from Lemma \ref{thmH} but formally we define it otherwise so our estimates are simpler.

\prt{Lemma}
\begin{proclaim}\label{thmtildeH}
Let $n\geq 3$, $\delta_k>0$, $\beta\geq n+1$ and $k\in\en$. Then we can find small enough $d_k>b_k>0$ and a bilipschitz mapping $\tilde{H}^S_k\colon Q(0,1)\to Q(0,1)$ such that $\tilde{H}^S_0(x)=x$ for every $x\in Q(0,1)$
$$
\tilde{H}^S_{k}(x)=\tilde{H}^S_{k-1}(x)\text{ for }x\notin \tilde{P}_k',\ \tilde{H}^S_{k}(x)=x\text{ for }
x\in Q(0,\hat{r}_k)
$$
$$
\text{ and }
\tilde{H}^S_k\text{ maps }\tilde{P}_k\text{ onto }P_k\text{ linearly}.
%H^S_k({T'}_k^S)=\tilde{T'}_k^S,
$$
%\eqn{aha}
%$$
%H^S_k(x)=[l_k(x_1),x_2,\hdots,x_n],\ l_k\text{ is linear on }[\hat{r}_i,\hat{r}_{i-1}],\
%i\in\{2,\hdots,k-1\}\text{ and }|l'_k|\leq 1\text{ there}.
%$$
%For each $k\in\en$ and $1\leq j \leq k$ we estimate its derivative as
Furthermore, %$|DH^S_k(x)|\leq 1$ for $x\in P_k^S$ and
\eqn{defhk2}
$$
\begin{aligned}
\int_{\tilde{P}_k'^S} |D\tilde{H}^S_k(x)|^{n-1}\; dx&\leq \delta_k.
\end{aligned}
$$
\end{proclaim}
\begin{proof}
This proof is analogous to the proof of Lemma \ref{thmH} and hence we skip some details.
We set $\tilde{H}^S_0(x)=x$ and we define
\eqn{ee2}
$$
\tilde{H}_k^S(x)=\tilde{H}_{k-1}^S(x)\text{ for }x\notin \tilde{P}_k'.
$$
%as identity on the whole $Q(0,\hat{r}_k)$ and further we define
%and it remains to define it on $P_k'$.
%Since $H_{k-1}$ was identity on
%$\{x_1\leq \hat{r}_{k-1}\}$ and $P_k'\cap\{x_1\geq \hat{r}_{k-1}\}\subset P_{k-1}$ (see \eqref{choicebd}) where $H_{k-1}^S$  is linear we obtain that
Then on $\partial \tilde{P}_k'$ we have
\eqn{eee2}
$$
\begin{aligned}
&\tilde{H}^S_k(x)=[\tilde{l}_{k-1}(x_1),x_2,\hdots,x_n], \text{ where }
\tilde{l}_{k-1}(x)=x_1\text{ for }x_1\leq \hat{r}_{k-1}\text{ and }\\
&\text{ for }x_1\in[\hat{r}_{k-1},\tilde{c}_k]\text{ it is linear with }\tilde{l}_{k-1}(\hat{r}_{k-1})=\hat{r}_{k-1}
\text{ and }\tilde{l}_{k-1}(\tilde{a}_{k-1})=a_{k-1}. \\
\end{aligned}
$$
%As $\tilde{a}_{k-1}<a_{k-1}$ we know that the derivative $|l'_{k-1}|\leq 1$ there.

Further, we define it for $x\in\{\tilde{a}_k\}\times [-d_k,d_k]^{n-1}$ as
%identity in the last $n-1$ coordinates and in the first coordinate as
$$
\begin{aligned}
\tilde{H}^S_k(x)&=[\tilde{\phi}_k(x),x_2,\hdots,x_n]\text{ where }\\
\tilde{\phi}_k(x)&:=\tilde{l}_{k-1}(\tilde{a}_k)+\Bigl(\log\log \frac{1}{\max\{b_k, |[x_2,\hdots,x_n]|_{\infty}\}}-\log\log\frac{1}{d_k}\Bigr). \\
\end{aligned}
$$
% where $|[x_2,\hdots,x_n]|_{\infty}=\max\{x_2,\hdots,x_n\}$. %IT'S REPETITION.
%Then it is easy to see $(H_k^s(x))_1=l_{k-1}(a_k)$ when $x_1=a_k$ and $|[x_2,\hdots,x_n]|_{\infty}=d_k$ and thus it agrees with \eqref{eee2} there. Moreover,
We fix $d_k$ small enough so that ($C_{\eqref{fixd2}}$ is a constant whose exact value we specify later)
\eqn{fixd2}
$$
\frac{2^{(\beta+1)k(n-1)}}{\log^{n-2}\frac{1}{d_k}}<C_{\eqref{fixd2}}\delta_k
$$
and we fix $b_k<d_k$ so that
\eqn{fff2}
$$
\text{for } |[x_2,\hdots,x_n]|_{\infty}=b_k \text{ we have }
\tilde{\phi}_k(x)=\tilde{l}_{k-1}(\tilde{a}_k)+\Bigl(\log\log \frac{1}{b_k}-\log\log\frac{1}{d_k}\Bigr)=a_k.
$$
%For every $x\in P_k$ we have $|[x_2,\hdots,x_n]|_{\infty}\leq b_k$ and thus $\phi_k(x)=\tilde{a}_k$. Therefore
For every $x\in \tilde{P}_k$ we can now define
\eqn{aaa2}
$$
\tilde{H}_k^S(x)=\bigl[\tilde{l}_{k}(x_1),x_2,\hdots,x_n\bigr]
\text{ where }\tilde{l}_k\text{ is linear with } \tilde{l}_{k}(\hat{r}_{k})=\hat{r}_{k}\text{ and }\tilde{l}_{k}(\tilde{a}_{k})=a_{k}.
$$
%It is easy to see that $|DH_k^S|\leq 1$ there and that this agrees with \eqref{eee2} used for $k-1$  before.
Finally on the hyperplane $x\in\{\hat{r}_{k-1}\}\times [-d_k,d_k]^{n-1}$ we define it as
%identity in the last $n-1$ coordinates and in the first coordinate as
$$
\begin{aligned}
\tilde{H}_k^S(x)&=[\tilde{\psi}_k(x),x_2,\hdots,x_n]\text{ where }\\
\tilde{\psi}_k(x)&:=\tilde{l}_{k-1}(\hat{r}_{k-1})+\tilde{A}_k\Bigl(\log\log \frac{1}{\max\{b_k, |[x_2,\hdots,x_n]|_{\infty}\}}-\log\log\frac{1}{d_k}\Bigr). \\
\end{aligned}
$$
%As before it agrees with \eqref{eee2} for $x_1=\hat{r}_{k-1}$ and $|[x_2,\hdots,x_n]|_{\infty}=d_k$.
The constant $\tilde{A}_k$ is chosen so that for $x\in \tilde{P}_k\cap\{x_1=\hat{r}_{k-1}\}$, i.e.\ for $|[x_2,\hdots,x_n]|_{\infty}\leq b_k$, we have %it agrees with \eqref{aaa2}.
$$
\tilde{l}_{k-1}(\hat{r}_{k-1})+\tilde{A}_k\Bigl(\log\log\frac{1}{b_k}-\log\log\frac{1}{d_k}\Bigr)=\frac{a_k+c_k}{2}.
$$
By this and \eqref{fff2} we obtain
$$
1\geq \tilde{l}_{k-1}(\hat{r}_{k-1})+\tilde{A}_k\Bigl(\log\log\frac{1}{b_k}-\log\log\frac{1}{d_k}\Bigr)
=\tilde{l}_{k-1}(\hat{r}_{k-1})+\tilde{A}_k\Bigl(a_k-\tilde{l}_{k-1}(\tilde{a}_k)\Bigr)
$$
and hence
$$
\tilde{A}_k\leq \frac{1}{a_k-\tilde{l}_{k-1}(\tilde{a}_k)}= \frac{1}{a_k-\tilde{a}_k} \leq C.
$$
%and so $A$ is bounded by a constant independent of $k$.

For every $x\in [\hat{r}_k,\tilde{c}_k]\times[-d_k,d_k]^{n-1}$ %we use linear interpolation between values on four hyperplanes ($x_1=\hat{r}_k$, $x_1=\hat{r}_{k-1}$, $x_1=a_k$ and $x_1=c_k$) with the help of function $h$ from Section \ref{sec:tentacle} and
we define
$$
\tilde{H}^S_k(x)=\Bigl[h\bigl(x_1;[\hat{r}_{k},\hat{r}_{k}];[\hat{r}_{k-1},\tilde{\psi}_k(x)],[\tilde{a}_k,\tilde{\phi}_k(x)],[\tilde{c}_k,\tilde{l}_{k-1}(\tilde{c}_k)]\bigr),x_2,\hdots,x_n\Bigr]
\text{ for }x\in \tilde{P}_k'.
$$
Again $\tilde{H}_k^S$ is bilipschitz on $Q(0,1)$.
%It remains to estimate the integrability of the derivative.
By \eqref{estder} we estimate the derivative with respect to first coordinate
$$
|D_1\tilde{H}_k^S(x)|\leq \begin{cases}
    \frac{\tilde{\psi}_k(x)-\hat{r}_k}{\hat{r}_{k-1}-\hat{r}_k},  & \text{ for } \hat{r}_k< x_1 < \hat{r}_{k-1}, \\
    \frac{\tilde{\phi}_k(x)-\tilde{\psi}_k(x)}{\tilde{a}_k-\hat{r}_{k-1}},  & \text{ for } \hat{r}_{k-1}< x_1 < \tilde{a}_k, \\
    \frac{l_{k-1}(\tilde{c}_k)-\tilde{\phi}_k(x)}{\tilde{c}_k-\tilde{a}_k}, & \text{ if } \tilde{a}_k < x_1 < \tilde{c}_k.
\end{cases}
$$
%Since $\phi_k(x)$ takes values between $l_{k-1}(a_k)$ and $\tilde{a}_k$ (see \eqref{fff2}) and
%$\psi_k$ takes values between $l_{k-1}(\hat{r}_{k-1})=\hat{r}_{k-1}$ and $\tilde{a}_k$ we can estimate this by universal constant $C$ (where $C$ does not depend on $\beta$ or $k$).
and this is clearly bounded by $C 2^{(\beta+1)k}$.
Furthermore, by \eqref{defh} and $\tilde{A}_k\leq C$ we know that we can estimate the derivative with respect to other coordinates by
%the constant multiple of the corresponding derivative of
%$$
%\frac{\psi_k(x)-\hat{r}_k}{\hat{r}_{k-1}-\hat{r}_k}+\frac{\phi_k(x)-\psi_k(x)}{a_k-\hat{r}_{k-1}}+\frac{l_{k-1}(c_k)%-\phi_k(x)}{c_k-a_k}.
%$$
%Since $A\leq C$ we can estimate this by
$$
C \max\Bigl\{\frac{1}{\hat{r}_{k-1}-\hat{r}_k},\frac{1}{\tilde{a}_k-\hat{r}_{k-1}},\frac{1}{\tilde{c}_k-\tilde{a}_k}\Bigr\}
\Bigl|D\Bigl(\log\log \frac{1}{\max\{b_k, |[x_2,\hdots,x_n]|_{\infty}\}}\Bigr)\Bigr|
$$
The maximum of the three terms can be estimated by $C \frac{1}{\hat{r}_k}\leq C 2^{(\beta+1) k}$
%and thus
%we can estimate the derivative with respect to other coordinates as
%$$
%|D_jH_k^s(x)|\leq
%\begin{cases}
%\frac{C 2^{\beta k}}{|[x_2,\hdots,x_n]|_{\infty}\log\frac{1}{|[x_2,\hdots,x_n]|_{\infty}}}&\text{ for %}b_k<|[x_2,\hdots,x_n]|_{\infty}<d_k,\\
%0&\text{ for }|[x_2,\hdots,x_n]|_{\infty}<b_k.\\
%\end{cases}
%$$
%Now
an a simple change to polar/spherical coordinates in $\R^{n-1}$ and \eqref{fixd2} gives us
$$
\begin{aligned}
\int_{\tilde{P}_k'} |D\tilde{H}_k^S(x)|^{n-1}\; dx&\leq C 2^{(\beta+1)k(n-1)}
\int_{\tilde{P}_k'} \frac{1}{|[x_2,\hdots,x_n]|^{n-1}_{\infty}\log^{n-1}\frac{1}{|[x_2,\hdots,x_n]|_{\infty}}}\; dx\\
&\leq C 2^{(\beta+1)k(n-1)}\int_0^{d_k}\frac{1}{r^{n-1}\log^{n-1}\frac{1}{r}}r^{n-2}\; dr\\
&\leq C 2^{(\beta+1)k(n-1)}\frac{1}{\log^{n-2}\frac{1}{d_k}}<C C_{\eqref{fixd2}}\delta_k <\delta_k,
\\
\end{aligned}
$$
where we have chosen $C_{\eqref{fixd2}}$ in \eqref{fixd2} so that the last inequality holds.
\end{proof}

Analogously to Theorem \ref{thmH2} we now obtain:

\prt{Theorem}
\begin{proclaim}\label{thmtildeH2}
Let $n\geq 3$, $\tilde{\delta}_k>0$, $\beta\geq n+1$ and $k\in\en$. Then we can find small enough $d_k>b_k>0$ and a bilipschitz mapping
 $\tilde{h}_k\colon Q(0,1)\to Q(0,1)$ such that $\tilde{h}_0(x)=x$ for every $x\in Q(0,1)$,
 \eqn{www2}
$$
\tilde{h}_{k+1}(x)=\tilde{h}_k(x)\text{ for }
x\notin \bigcup_{\hat{\ve}(k)\in\hat{\mathbb{V}}^{k}} \tilde{P}_{\hat{\ve}(k)}',\ \tilde{h}_k(x)=x\text{ for }x\in \hat{Q}_{\hat{\ve}(k)}
\text{ and } \tilde{h}_k(\tilde{P}_{\hat{\ve}(k)})=P_{\hat{\ve}(k)}.
%\ h_k({T}'_{\ve(k)})=\tilde{T}'_{\ve(k)}.
$$
% $$
% \text{ and } J_{h_k}(x) > 0 \text{ for all } x \in Q(0,1).
% $$
We can estimate the integral of its derivative as
$$
\int_{\bigcup_{\hat{v}(k)\in \hat{\mathbb{V}}^k} \tilde{P}^{'}_{\hat{\ve}(k)}} |D\tilde{h}_k(x)|^{n-1}\; dx\leq \tilde{\delta}_k.
$$

Moreover, a pointwise limit $\tilde{h}$ of $\tilde{h}_k$ is continuous and one-to-one on $Q(0,1)$ and
$J_{\tilde{h}}(x) > 0$ a.e. And, there is a continuous $\tilde{t}:Q(0,1)\to Q(0,1)$ which is a generalized inverse to
$\tilde{h}$, i.e.\ $\tilde{t}(\tilde{h}(x))=x$ for every $x\in [-1,1]^n$. On the other hand,
\eqn{ppp}
$$
\tilde{t}(l_x)=x\text{ for every }x\in C_B^T,
$$
where
$l_x$ is defined by \eqref{def:l}.
\end{proclaim}
\begin{proof}
The proof of the next theorem is analogous to the proof of Theorem \ref{thmH2} and therefore we skip it. We only explain why \eqref{ppp} holds.

Outside of $\bigcup_{\ve(k)\in\hat{\mathbb{V}}^{k}} \tilde{T}_{\hat{\ve}(k)}'$ all mappings $\tilde{h}_l$, $l\geq k$, are equal to $\tilde{h}_{k-1}$ and hence they are bilipschitz there and $J_{\tilde{h}_l}>0$ a.e. It follows that we can define $\tilde{h}=\lim_{k\to\infty} \tilde{h}_k$ everywhere outside of
$$
\bigcap_{k=1}^{\infty}\bigcup_{\hat{\ve}(k)\in\hat{\mathbb{V}}^{k}} \tilde{T}_{\hat{\ve}(k)}'=C_B^T
$$
and it is one-to-one and continuous there with $J_{\tilde{h}}>0$ a.e. For $x\in C_B^T$ we define $\tilde{h}(x)=x$ and notice that now $\tilde{h}$ is one-to-one everywhere. %Let us also not that this gives us a quasicontinuous representative of $\tilde{h}$. EXPLAIN??

We define $\tilde{t}=\tilde{h}^{-1}$ on $Q(0,1)\setminus \tilde{h}(C_B^T)$ and notice that $\tilde{t}$ is continuous there.
Since
$$h^{-1}_k(T_{\hat{\ve}(k)})=\tilde{T}_{\hat{\ve}(k)}\text{ and }\operatorname{diam} \tilde{T}_{\hat{\ve}(k)}\to 0
$$
it is not difficult to see that for every $a\in l_x:=\bigcap_{k=1}^{\infty}T_{\hat{\ve}(k)}$ we can define
$\tilde{t}(a)=x$ and now $\tilde{t}$ is continuous everywhere. For $x\in C_B^T$ we have
$\tilde{t}(\tilde{h}(x))=\tilde{t}(x)=x$ and hence $\tilde{t}$ is a generalized inverse to $\tilde{h}$.
\end{proof}

\subsection{Counterexample in Theorem \ref{secondthm}}

\begin{proof}[Construction of the counterexample in Theorem \ref{secondthm}]
Again we use the same sequences
$$
\alpha_k = \frac{1}{2}\left(1+2^{-k\beta}\right)\text{ and }\beta_k = 2^{-k \beta}\text{ with }\beta\geq n+1
$$
to define Cantor type sets $C_A$, $C_B$ and $C_B^T$. As in the proof of Theorem \ref{firstthm} we have the estimates of the derivatives \eqref{dg-1} and \eqref{dg-2}.
We set
\eqn{defdeltak2}
$$
\tilde{\delta}_k=\frac{2^{-k\beta (2n-1)}}{k^2}.
$$
Given this $\tilde{\delta}_k$ we find $d_k>b_k>0$ so that we have Theorem \ref{thmtildeH2}.

Consider the mapping $\tilde{f}$ as a pointwise limit of
$$
\f_k(y) = g_k^{-1} \circ L^{-1} \circ \tilde{h}_k \circ L \circ g_k(y)
$$
almost everywhere (see Fig.~\ref{pic:the_mapping-1}).
For $y\in C_A$ we know that $L \circ g(y)\in C_B^T$ where $\tilde{h}_k(x)=x$ and hence it is easy to see that the pointwise limit is equal to $\f(y)=y$ for $y\in C_A$.
Therefore, we see at once that the pointwise limit of
$\f_k$ is
$$
\f(y)=g^{-1} \circ L^{-1} \circ \tilde{h} \circ L \circ g(y)\text{ everywhere}.
$$
Since $g$ and $L$ are homeomorphisms and $\tilde{h}$ is one-to-one we obtain that $\f$ is one-to-one on $Q(0,1)$.
It is not difficult to see that $\f$ is locally bilipschitz on
$[-1,1]^n\setminus C_A$ and hence we can use the composition formula for derivatives to obtain
$$
J_{\f}(y)=J_{g^{-1}}J_{L^{-1}} J_{\tilde{h}}J_L J_g >0\text{ for a.e. }x\in[-1,1]^n\setminus C_A.
$$
For $y\in C_A$ we know that $\f(y)=y$ and hence $J_{\f}=1$ for a.e. $x\in C_A$ once we show that $\f\in W^{1,1}$ since the weak derivative is equal to the approximative derivative a.e.

With the help of  Theorem \ref{thmtildeH2} we obtain that the continuous mapping
$$
w(y)=g^{-1} \circ L^{-1} \circ \tilde{t} \circ L \circ g(y)
$$
is a generalized inverse to $\f$. Moreover, for every $y\in C_A$ we know that $x=L \circ g(y)\in C_B^T$.
Therefore, the standard arguments show that for $\tilde{l}_x=(L\circ g)^{-1}(l_x)$ we have by \eqref{ppp}
$$
w(\tilde{l}_x)=g^{-1} \circ L^{-1} \circ \tilde{t} \circ L \circ g(\tilde{l}_x)
=g^{-1} \circ L^{-1} \circ \tilde{t}(l_x)=
g^{-1} \circ L^{-1}(x)=y.
$$
Now $\tilde{l}_x$ is a continuum and so is $w^{-1}(y)$ for every $y\in C_A$.

\begin{figure}[h]
\begin{center}
\begin{tikzpicture}[line cap=round,line join=round,>=triangle 45,x=0.36cm,y=0.36cm]
\clip(-26,-4) rectangle (8,18);

% \draw (-25,-3) -- (-25,5) -- (-17,5) -- (-17,-3) -- cycle;
\draw (-25,17)-- (-25,9)-- (-17,9)-- (-17,17) -- cycle;
\draw (-24.5,13.5) -- (-24.5,16.5) -- (-21.5,16.5) -- (-21.5,13.5) -- cycle;
\draw (-21.5,15.5) -- (-20.7,13.3) -- (-20.7,13.1) -- (-21.5,14.5) -- cycle;
\fill[fill=black,fill opacity=0.1] (-21.5,15.5) -- (-20.7,13.3) -- (-20.7,13.1) -- (-21.5,14.5) -- cycle;
\draw (-20.5,13.5) -- (-20.5,16.5) -- (-17.5,16.5) -- (-17.5,13.5) -- cycle;
\draw (-17.5,14.5) -- (-17.2,15) -- (-17.2,16.7) -- (-18.8,16.7) -- (-18.7,16.8) -- (-18.9,16.8) -- (-19.2,16.6) -- (-17.4,16.6) -- (-17.4,15.8) -- (-17.5,15.5) -- cycle;
\fill[fill=black,fill opacity=0.1] (-17.5,14.5) -- (-17.2,15) -- (-17.2,16.7) -- (-18.8,16.7) -- (-18.7,16.8) -- (-18.9,16.8) -- (-19.2,16.6) -- (-17.4,16.6) -- (-17.4,15.8) -- (-17.5,15.5) -- cycle;
\draw (-20.5,9.5) -- (-20.5,12.5) -- (-17.5,12.5) -- (-17.5,9.5) -- cycle;
\draw (-17.5,11.5) -- (-17.2,11) -- (-17.2,9.3) -- (-18.8,9.3) -- (-18.7,9.2) -- (-18.9,9.2) -- (-19.2,9.4) -- (-17.4,9.4) -- (-17.4,10.6) -- (-17.5,10.5) -- cycle;
\fill[fill=black,fill opacity=0.1] (-17.5,11.5) -- (-17.2,11) -- (-17.2,9.3) -- (-18.8,9.3) -- (-18.7,9.2) -- (-18.9,9.2) -- (-19.2,9.4) -- (-17.4,9.4) -- (-17.4,10.6) -- (-17.5,10.5) -- cycle;
\draw (-24.5,9.5) -- (-24.5,12.5) -- (-21.5,12.5) -- (-21.5,9.5) -- cycle;
\draw (-21.5,10.5) -- (-20.7,12.7) -- (-20.7,12.9) -- (-21.5,11.5) -- cycle;
\fill[fill=black,fill opacity=0.1] (-21.5,10.5) -- (-20.7,12.7) -- (-20.7,12.9) -- (-21.5,11.5) -- cycle;

% \draw (-13,-3) -- (-13,5) -- (-5,5) -- (-5,-3) -- cycle;
\draw (-13,17)-- (-13,9)-- (-5,9)-- (-5,17) -- cycle;
\draw (-11.3,15.3) -- (-11.3,14.7) -- (-10.7,14.7) -- (-10.7,15.3) -- cycle;
\draw (-10.7,14.9) -- (-7.8,13.9) -- (-7.8,14.1) -- (-10.7,15.1) -- cycle;
\fill[fill=black,fill opacity=0.1] (-10.7,14.9) -- (-7.8,13.9) -- (-7.8,14.1) -- (-10.7,15.1) -- cycle;
\draw (-6.7,15.3) -- (-6.7,14.7) -- (-7.3,14.7) -- (-7.3,15.3) -- cycle;
\draw (-6.7,15.1) -- (-6.6,15.3) -- (-7.5,16.1) -- (-7.8,16.3) -- (-7.3,16.6) -- (-7.3,16.5) -- (-7.5,16.3) -- (-6.4,15.3) -- (-6.4,15.2) -- (-6.7,14.9) -- cycle;
\fill[fill=black,fill opacity=0.1] (-6.7,15.1) -- (-6.6,15.3) -- (-7.5,16.1) -- (-7.8,16.3) -- (-7.3,16.6) -- (-7.3,16.5) -- (-7.5,16.3) -- (-6.4,15.3) -- (-6.4,15.2) -- (-6.7,14.9) -- cycle;
\draw (-6.7,10.7) -- (-6.7,11.3) -- (-7.3,11.3) -- (-7.3,10.7) -- cycle;
\draw (-6.7,10.9) -- (-6.6,10.7) -- (-7.5,9.9) -- (-7.8,9.7) -- (-7.3,9.4) -- (-7.3,9.5) -- (-7.5,9.7) -- (-6.4,10.7) -- (-6.4,10.8) -- (-6.7,11.1) -- cycle;
\fill[fill=black,fill opacity=0.1] (-6.7,10.9) -- (-6.6,10.7) -- (-7.5,9.9) -- (-7.8,9.7) -- (-7.3,9.4) -- (-7.3,9.5) -- (-7.5,9.7) -- (-6.4,10.7) -- (-6.4,10.8) -- (-6.7,11.1) -- cycle;
\draw (-11.3,10.7) -- (-11.3,11.3) -- (-10.7,11.3) -- (-10.7,10.7) -- cycle;
\draw (-10.7,11.1) -- (-7.8,12.1) -- (-7.8,11.9) -- (-10.7,10.9) -- cycle;
\fill[fill=black,fill opacity=0.1] (-10.7,11.1) -- (-7.8,12.1) -- (-7.8,11.9) -- (-10.7,10.9) -- cycle;

\draw (-1,17)-- (-1,9)-- (7,9)-- (7,17) -- cycle;
\draw (3.3,15.7) -- (2.7,15.7) -- (2.7,16.3) -- (3.3,16.3) -- cycle;
\draw (3.3,16.1) -- (3.7,16.1) -- (3.7,15.9) -- (3.3,15.9) -- cycle;
\fill[fill=black,fill opacity=0.1] (3.3,16.1) -- (3.7,16.1) -- (3.7,15.9) -- (3.3,15.9) -- cycle;
\draw (2.7,14.3) -- (3.3,14.3) -- (3.3,13.7) -- (2.7,13.7) -- cycle;
\draw (3.3,13.9) -- (3.7,13.9) -- (3.7,14.1) -- (3.3,14.1) -- cycle;
\fill[fill=black,fill opacity=0.1] (3.3,13.9) -- (3.7,13.9) -- (3.7,14.1) -- (3.3,14.1) -- cycle;
\draw (3.3,11.7) -- (2.7,11.7) -- (2.7,12.3) -- (3.3,12.3) -- cycle;
\draw (3.3,12.1) -- (3.7,12.1) -- (3.7,11.9) -- (3.3,11.9) -- cycle;
\fill[fill=black,fill opacity=0.1] (3.3,12.1) -- (3.7,12.1) -- (3.7,11.9) -- (3.3,11.9) -- cycle;
\draw (2.7,10.3) -- (3.3,10.3) -- (3.3,9.7) -- (2.7,9.7) -- cycle;
\draw (3.3,9.9) -- (3.7,9.9) -- (3.7,10.1) -- (3.3,10.1) -- cycle;
\fill[fill=black,fill opacity=0.1] (3.3,9.9) -- (3.7,9.9) -- (3.7,10.1) -- (3.3,10.1) -- cycle;

\draw (-1,-3) -- (-1,5) -- (7,5) -- (7,-3) -- cycle;
\draw (2.7,3.7)-- (3.3,3.7)-- (3.3,4.3) -- (2.7,4.3)-- cycle;
\draw (3.3,3.9)-- (6.3,3.9)-- (6.3,4.1)-- (3.3,4.1) -- cycle;
\fill[fill=black,fill opacity=0.1] (3.3,3.9)-- (6.3,3.9)-- (6.3,4.1)-- (3.3,4.1) -- cycle;
\draw (3.3,2.3) -- (2.7,2.3) -- (2.7,1.7) -- (3.3,1.7) -- cycle;
\draw (3.3,2.1) -- (6.3,2.1) -- (6.3,1.9) -- (3.3,1.9) -- cycle;
\fill[fill=black,fill opacity=0.1] (3.3,2.1) -- (6.3,2.1) -- (6.3,1.9) -- (3.3,1.9) -- cycle;
\draw (2.7,-0.3) -- (3.3,-0.3) -- (3.3,0.3) -- (2.7,0.3) -- cycle;
\draw (3.3,-0.1) -- (6.3,-0.1) -- (6.3,0.1) -- (3.3,0.1) -- cycle;
\fill[fill=black,fill opacity=0.1] (3.3,-0.1) -- (6.3,-0.1) -- (6.3,0.1) -- (3.3,0.1) -- cycle;
\draw (3.3,-1.7) -- (2.7,-1.7) -- (2.7,-2.3) -- (3.3,-2.3) -- cycle;
\draw (3.3,-1.9) -- (6.3,-1.9) -- (6.3,-2.1) -- (3.3,-2.1) -- cycle;
\fill[fill=black,fill opacity=0.1] (3.3,-1.9) -- (6.3,-1.9) -- (6.3,-2.1) -- (3.3,-2.1) -- cycle;

% \draw (-13,17)-- (-13,9)-- (-5,9)-- (-5,17) -- cycle;
\draw (-13,-3) -- (-13,5) -- (-5,5) -- (-5,-3) -- cycle;
\draw (-11.3,3.3) -- (-11.3,2.7) -- (-10.7,2.7) -- (-10.7,3.3) -- cycle;
\draw (-10.7,3.1) -- (-8.8,2.6) -- (-5.4,2.1) -- (-5.4,1.9) -- (-9.2,2.5) -- (-10.7,2.9) -- cycle;
\fill[fill=black,fill opacity=0.1] (-10.7,3.1) -- (-8.8,2.6) -- (-5.4,2.1) -- (-5.4,1.9) -- (-9.2,2.5) -- (-10.7,2.9) -- cycle;
\draw (-6.7,3.3) -- (-6.7,2.7) -- (-7.3,2.7) -- (-7.3,3.3) -- cycle;
\draw (-6.7,3.1) -- (-5.4,4.1) -- (-5.4,3.9) -- (-6.7,2.9) -- cycle;
\fill[fill=black,fill opacity=0.1] (-6.7,3.1) -- (-5.4,4.1) -- (-5.4,3.9) -- (-6.7,2.9) -- cycle;
\draw (-6.7,-1.3) -- (-6.7,-0.7) -- (-7.3,-0.7) -- (-7.3,-1.3) -- cycle;
\draw (-6.7,-1.1) -- (-5.4,-2.1) -- (-5.4,-1.9) -- (-6.7,-0.9) -- cycle;
\fill[fill=black,fill opacity=0.1] (-6.7,-1.1) -- (-5.4,-2.1) -- (-5.4,-1.9) -- (-6.7,-0.9) -- cycle;
\draw (-11.3,-1.3) -- (-11.3,-0.7) -- (-10.7,-0.7) -- (-10.7,-1.3) -- cycle;
\draw (-10.7,-1.1) -- (-8.8,-0.6) -- (-5.4,-0.1) -- (-5.4,0.1) -- (-9.2,-0.5) -- (-10.7,-0.9) -- cycle;
\fill[fill=black,fill opacity=0.1] (-10.7,-1.1) -- (-8.8,-0.6) -- (-5.4,-0.1) -- (-5.4,0.1) -- (-9.2,-0.5) -- (-10.7,-0.9) -- cycle;

% \draw (-25,17)-- (-25,9)-- (-17,9)-- (-17,17) -- cycle;
\draw (-25,-3) -- (-25,5) -- (-17,5) -- (-17,-3) -- cycle;
\draw (-24.5,4.5)-- (-24.5,1.5)-- (-21.5,1.5)-- (-21.5,4.5) -- cycle;
\draw (-21.5,3.5) -- (-21,2.5) -- (-20.5,1.3) -- (-17.5,1.3) -- (-17.1,2.1) -- (-17.1,1.9) -- (-17.4,1.2) -- (-20.6,1.1) -- (-21.5,2.5) -- cycle;
\fill[fill=black,fill opacity=0.1] (-21.5,3.5) -- (-21,2.5) -- (-20.5,1.3) -- (-17.5,1.3) -- (-17.1,2.1) -- (-17.1,1.9) -- (-17.4,1.2) -- (-20.6,1.1) -- (-21.5,2.5) -- cycle;
\draw (-20.5,4.5)-- (-20.5,1.5)-- (-17.5,1.5)-- (-17.5,4.5) -- cycle;
\draw (-17.5,3.5) -- (-17.1,4.6) -- (-17.1,4.4) -- (-17.5,2.5) -- cycle;
\fill[fill=black,fill opacity=0.1] (-17.5,3.5) -- (-17.1,4.6) -- (-17.1,4.4) -- (-17.5,2.5) -- cycle;
\draw (-20.5,0.5)-- (-20.5,-2.5)-- (-17.5,-2.5)-- (-17.5,0.5) -- cycle;
\draw (-17.5,-1.5) -- (-17.1,-2.6) -- (-17.1,-2.4) -- (-17.5,-0.5) -- cycle;
\fill[fill=black,fill opacity=0.1] (-17.5,-1.5) -- (-17.1,-2.6) -- (-17.1,-2.4) -- (-17.5,-0.5) -- cycle;
\draw (-24.5,0.5)-- (-24.5,-2.5)-- (-21.5,-2.5)-- (-21.5,0.5) -- cycle;
\draw (-21.5,-1.5) -- (-21,-0.5) -- (-20.5,0.7) -- (-17.5,0.7) -- (-17.1,-0.1) -- (-17.1,0.1) -- (-17.4,0.8) -- (-20.6,0.9) -- (-21.5,-0.5) -- cycle;
\fill[fill=black,fill opacity=0.1] (-21.5,-1.5) -- (-21,-0.5) -- (-20.5,0.7) -- (-17.5,0.7) -- (-17.1,-0.1) -- (-17.1,0.1) -- (-17.4,0.8) -- (-20.6,0.9) -- (-21.5,-0.5) -- cycle;

\draw [->] (-16.3,13) -- (-13.7,13);
\draw [->] (-4.3,13) -- (-1.7,13);
\draw [->] (-1.7,1) -- (-4.3,1);
\draw [->] (3,8.5) -- (3,5.7);
\draw [->] (-13.7,1) -- (-16.3,1);
\begin{scriptsize}
\draw[color=black] (-15,13.7) node {$g_k$};
\draw[color=black] (-3,13.7) node {$L$};
\draw[color=black] (3.8,7.2) node {$h^{-1}_k$};
\draw[color=black] (-2.7,1.7) node {$L^{-1}$};
\draw[color=black] (-14.7,1.7) node {$g_k^{-1}$};
\end{scriptsize}
\end{tikzpicture}
\end{center}
\caption{Mapping $\f$.}
\label{pic:the_mapping-1}
\end{figure}

By Theorem \ref{thmtildeH2} we know that $\tilde{h}_{k-1}=\tilde{h}_k$ for every $y$ with
$L(g_k(y))\notin \bigcup_{\hat{\ve}(k)\in\hat{\mathbb{V}}^{k}} \tilde{P}_{\hat{\ve}(k)}'$ and
$g_{k-1}(y)=g_{k}(y)$ for $y\notin \bigcup_{\ve(k-1)\in\mathbb{V}^{k-1}} Q_{\ve(k-1)}$ by \eqref{star}.
In view of  \eqref{goodmap} it follows that
$$
\f_k(y)=\f_{k-1}(y)\text{ for }L(g_k(y))\notin \bigcup_{\hat{\ve}(k)\in\hat{\mathbb{V}}^{k}} \tilde{P}_{\hat{\ve}(k)}'=:\tilde{M}_k\text{ and }y\notin \bigcup_{\ve(k-1)\in\mathbb{V}^{k-1}} Q_{\ve(k-1)}.
$$
Note that for $x\in \bigcup_{\hat{\ve}(k-1)\in\hat{\mV}^{k-1}} \hat{Q}_{\hat{\ve}(k-1)}\setminus \tilde{M}_k$ we have by Theorem
\ref{thmtildeH2}
$$
\tilde{h}_{k}(x)=\tilde{h}_{k-1}(x)=x.
$$
In view of $g_k(Q_{\ve(k-1)})=g_{k-1}(Q_{\ve(k-1)})=\tilde{Q}_{\ve(k-1)}$  and
% $L_{k-1}(\tilde{Q}_{\ve(k-1)})=L_k(\tilde{Q}_{\ve(k-1)})$
$L(\tilde{Q}_{\ve(k-1)})=\hat{Q}_{\hat{\ve}(k-1)}$
we obtain by Theorem \ref{Bilip}  for $y\in Q_{\ve(k-1)}\setminus g_k^{-1}(L^{-1}(\tilde{M}_k))$ that
$$
\tilde{f}_{k-1}(y)=g_{k-1}^{-1} \circ L^{-1} \circ x \circ L \circ g_{k-1}(y)=y\text{ and similarly }
\tilde{f}_{k}(y)=y.
$$
Therefore
\eqn{prvni2}
$$
\begin{aligned}
\int_{Q(0,1)}|D\f_{k}-D\f_{k-1}|^{n-1}&=\int_{g_k^{-1}(L^{-1}(\tilde{M}_k))}|D\f_{k}-D\f_{k-1}|^{n-1}\\
&\leq C \int_{g_k^{-1}(L^{-1}(\tilde{M}_k))}|D\f_{k}|^{n-1}+
C\int_{g_k^{-1}(L^{-1}(\tilde{M}_{k}))}|D\f_{k-1}|^{n-1}. \\
\end{aligned}
$$

%The same as before, $h^{-1}_k(x)=x$ for every $y\in \hat{Q}^B_{k} $ and hence $\f_k(y)=y$ for every $y\in Q^A_{k}$.
%It follows that
%$D \f_k(y)=1$ there,
%moreover, $\f_k \in W^{1,1}_{\loc}$.
%It remains to estimate the derivative over $Q_0\setminus C_A$.

Note that $\f_k$ is bilipschitz (as a composition of bilipschitz mappings) and hence we can compute its derivative a.e.\ by the composition of derivatives. With the help of \eqref{dl} we get
$$
\begin{aligned}
|D\f_k(y)|&\leq |Dg_k^{-1}|\cdot |DL^{-1}|\cdot |D\tilde{h}_k|\cdot |DL|\cdot |Dg_k|\\
&\leq C \bigl|Dg^{-1}_k(L^{-1}\circ \tilde{h}_k\circ L\circ g_k(y))\bigr|\cdot \bigl|D\tilde{h}_k(L\circ g_k(y)))\bigr|.\\
\end{aligned}
$$
%whenever $ L^{-1}\circ h_k$.
By the change of variables
%and
%$L_k\circ g_k(Q^A_k)=\hat{Q}^B_k$
% $L_k\circ g_k(C_A)=C_B^T$
\eqn{ahoj2}
$$
\begin{aligned}
\int_{g_k^{-1}(L^{-1}(\tilde{M}_k))} &|D\f_k(y)|^{n-1}\; dy \leq
C\int_{g_k^{-1}(L^{-1}(\tilde{M}_k))} |Dg_k^{-1}|^{n-1} |D\tilde{h}_k|^{n-1}\frac{J_{L}}{J_{L}} \frac{J_{g_k}}{J_{g_k}} \; dy\\
&\leq C\int_{\tilde{M}_k} |Dg_k^{-1}(L^{-1}\circ \tilde{h}_k(x))|^{n-1} |D\tilde{h}_k(x)|^{n-1}\frac{1}{J_{g_k}((L\circ g_k)^{-1}(x))}\; dx.\\
\end{aligned}
$$
%We divide the set
%$$
%Q_0\setminus \hat{Q}^B_k=\Bigl(Q_0\setminus \bigcup_{\hat{v}(1)\in\hat{\mathbb{V}}(1)}\tilde{T}'_{\hat{v}(1)}\Bigr)
%\cup\Bigl(\bigcup_{i=1}^{k}\bigcup_{\hat{v}(i)\in\hat{\mathbb{V}}(i)}\bigl(\tilde{T}'_{\hat{v}(i)}\setminus \tilde{T}_{\hat{v}(i)}\bigr)\Bigr)
%\cup \bigcup_{\hat{v}(k)\in\hat{\mathbb{V}}(k)}\bigl(\tilde{T}_{\hat{v}(k)}\setminus \hat{Q}^B_{\hat{v}(k)} \bigr)=:S_1\cup S_2\cup S_3
%$$

Note that for every $x\in \tilde{P}_{\hat{\ve}(k)}'\subset \tilde{M}_k$ we know that $L^{-1}\circ \tilde{h}_k(x)$ lies outside of $\bigcup_{\ve(k)\in\mV^k} \tilde{Q}_{\ve(k)}$ and hence we can use \eqref{eq:Dg2} to estimate
$$
|Dg_k^{-1}(L^{-1}\circ \tilde{h}_k(x))|\leq
C\max_{i=1,\hdots,k} 2^{\beta i}=C 2^{\beta k}
$$
and
$$
\frac{1}{J_{g_k}((L\circ g_k)^{-1}(x))}\leq C 2^{\beta k n}.
$$
Now \eqref{defdeltak2} and \eqref{ahoj2} imply that
$$
\int_{g_k^{-1}(L^{-1}(\tilde{M}_k))} |D\f_k(y)|^{n-1}\; dy \leq
C2^{k\beta(2n-1)}\int_{\tilde{M}_k} |D\tilde{h}_k(x)|^{n-1}\; dx\leq \frac{C}{k^2}.
$$
%Note that for $y\in \bigcup_{\ve(k)\in\mathbb{V}^{k}} Q_{\ve(k)}$ we have $\tilde{h}_{k-1}=\tilde{h}_k(x)$ by Theorem
%\ref{thmtildeH2} and hence $\tilde{f}_{k-1}(y)=\tilde{f}_k(y)$.
%It remains to consider
%$y\in Q_{\ve(k-1)}\setminus \bigcup_{\ve(k)\in\mathbb{V}^{k}} Q_{\ve(k)}$. WHAT NOW???
The similar estimate holds also for $D\f_{k-1}$ and hence \eqref{prvni2} implies that
$$
\int_{Q(0,1)}|D\f_{k}-D\f_{k-1}|^{n-1}\leq \frac{C}{k^2}.
$$
Since $1/k^2$ is a convergent series, $f_k$ form a Cauchy sequence in $W^{1,n-1}$ and hence $f\in W^{1,n-1}$.
\end{proof}

\prt{Example}
\begin{proclaim}\label{last}
For every $n\geq 2$ there is a set $C_A$ of Hausdorff dimension $n$ and a Lipschitz mapping $f_L\colon[-1,1]^n\to [-1,1]^n$ with $J_{f_L}>0$ a.e.\ which is a strong limit of Sobolev homeomorphisms $f_k\in W^{1,n-1}([-1,1]^n,\rn)$ with $f_k(x)=x$ for $x\in\partial[-1,1]^n$ such that
$$
f(C_A)\text{ is a point}.
$$
\end{proclaim}
\begin{proof}
We only briefly sketch the construction. We set $\alpha_k=\frac{1}{k}$ in the construction of a Cantor type set $C_A$ (see Section~\ref{ssec:CS}). Then it is easy to see that the measure of $C_A$ is zero but its Hausdorff dimension is $n$. We map this by $g$ from Section~\ref{sec:map_CS} to a Cantor type set $C_B$ given by sequence $\beta_k= 2^{-\beta k}$, $\beta\geq n+1$, as usual. Note that by \eqref{eq:Dg},
$$
\frac{\beta_k}{\alpha_k}\leq C\text{ and }\frac{\beta_{k-1} -\beta_{k}}{\alpha_{k-1} - \alpha_{k}}\leq C
$$
we obtain that $g$ is a Lipschitz mapping.

Then we map $C_B$ by the Lipschitz mapping $L$ from Theorem \ref{Bilip} to the Cantor tower $C_B^T$.
Now $C_B^T\subset \{0\}^{n-1}\times(-1,1)$ and it is easy to find a Lipschitz mapping $S$ which squeezes a segment containing $C_B^T$ to a single point, it is one-to-one outside of this segment and equals to identity on $\partial [-1,1]^n$.
We can choose $S$ to be
$$
S(x)=\Bigl[x_1,x_2,\hdots,x_{n-1},x_n\sqrt{x_1^2+\hdots+x_{n-1}^2}\Bigr]
$$
on $Q(0,1-\delta)$ (fix $\delta>0$ so that $C_B^T\subset Q(0,1-\delta)$) and extend it in a Lipschitz way so that $S(x)=x$ on $\partial Q(0,1)$.
Finally the mapping $f_L:=S\circ L\circ g$ is a mapping for which
$$
f_L(C_A)=S(C_B^T)\text{ is a point }
$$
and we can obtain it as a weak limit of homeomorphisms in $W^{1,\infty}$ (or even strong limit in $W^{1,p}$ for any $p<\infty$).
\end{proof}

%\subsection{Lemma 3.2 for $p=n-1$ and the inverse problem}

%For the inverse problem (we map a point to a continuum) we use essentially the inverse of the mapping from the lemma above. The key estimates are that
%$$
%|D_1 \tilde{H}_k|\leq C(a_k,c_k)
%$$
%where the constant $C(a_k,c_k)$ is really big and it depends on values of $a_k$ and $c_k$ (and $a_{k-1}$, $c_k-1$...). %However on our set of measure $d_k^{n-1}$ (or even $d_k^{n-1} \tilde{c}_k$) we can make this integral as small as we %wish.

%Note that in \eqref{defh} we have $\frac{\tilde{c}-\tilde{a}}{c-a} t$ so the unpleasant term like %$\frac{1}{\tilde{c}_k-\tilde{a}_k}$ will appear also for the estimate of $D_j H_k$. However we can bound them by some
%$C(a_k,c_k)$ (big constant but depends only on $a_k$ and $c_k$ but definitely not on $d_k$). As above
%$$
%\int_{\tilde{T}_k'^S} |D\tilde{H}_k^s(x)|^{n-1}\; dx\leq C(a_k,c_k)\frac{1}{\log^{n-2}\frac{1}{d_k}}
%$$
%and again we can make this as small as possible by our choice of $d_k$.

%~~~~~~~~~~~~~~~~~~~~~~~~~~~~~~~~~~~~~~~~~~~~~~~~~~~~~~~~~~~~~~~~~~~~~~~~~~~~~~

\section{Positive statements: the case $p>n-1$}
%{\color{brown}
%\begin{itemize}
	%\item IN MÜLLER-SPECTOR THEY HAVE A DIFFERENT DEFFINITION OF THE MULTIFUNCTION $f^T$. THEY DEFINE
	%$$
	%E(f,B(a,r)):=  f^T(B(a,r))\cup f(S(a,r)),
	%$$
	%AND
	%$$
	%f^T(x) := \bigcap_{r>0,\ r\notin N_a} E(f^\ast, B(a,r)),
	%$$
	%WHERE $f^\ast$ IS THE PRECISE REPRESENTATIVE.
%\end{itemize}
%}

To study the injectivity a.e.\ with respect to the image we define slightly better (INV) condition, see Corollary \ref{beterRepresentativeIsBetter} below. We need the following generalization of \cite[Lemma~7.3]{MS} for the case with no additional assumptions on $J_f$.
% The following lemma will allow us to do that. It is \cite[Lemma 7.3]{MS} without the assumption $J_f\neq 0$ a.e.

\begin{lemma}\label{telescopic}
Let $f\in W^{1,p}(\Omega,\R^n)$, $p>n-1$, be a weak limit of homeomorphisms $f_k$ in $W^{1,p}(\Omega,\R^n)$, and $a$, $b\in \Omega$.
Then there exist $\cL^1$-null sets $N_a$ and $N_b$ such that for every $r\in (0,r_a)\setminus N_a$ and $s\in (0,r_b)\setminus N_b$ \textrm{(}where $r_x:=\dist(x,\partial\Omega)$\textrm{)} the following holds:
\begin{enumerate}[(i)]
\item If $B(a,r)\subset B(b,s)$, then
$$
E(f^\ast,B(a,r))\subset E(f^\ast,B(b,s)).
$$
\item If $B(a,r) \cap B(b,s) = \varnothing$, then
$$
{f^\ast}^T( B(a,r)) \cap {f^\ast}^T( B(b,s)) = \emptyset.
$$
\end{enumerate}
\end{lemma}
\begin{proof}
We may assume that $f$ equals to the representative $f^\ast$.
By Lemma~\ref{lem:MS2.9},
there are $\cL^1$-null sets $N_a$ and $N_b$ such that for every
$r\in (0,r_a)\setminus N_a$ and $s\in (0,r_b)\setminus N_b$
one has $f_k\to f$ (up to subsequence) uniformly on $S(a,r)$ and $S(b,s)$.

To establish (i), we show that $\deg(f,S(b,s),y) \neq 0$
for $y\in E(f,B(a,r))\setminus f(S(b,s))$.
% (otherwise the theorem clearly holds).
% For (i), consider a point $y\in E(f,B(a,r))$.
% It suffices to assume that $y\notin f(S(b,s))$, otherwise the theorem clearly holds.
% We need to show that $\deg(f,S(b,s),y)$ is not equal to zero.
Let us firstly suppose that $y=f(x)$ for $x\in S(a,r)$.
%(and because od \eqref{neNaSfere} $x\notin S(b,s)$).
Since $f(S(b,s))$ is compact and $f_k$ converge uniformly
% Due to compactness of $f(S(b,s))$ and the uniform convergence of $f_k$
on the sphere $S(b,s)$ there exist $\epsilon>0$ and $k_0\in\mathbb{N}$ such that for every $k>k_0$ we have $B(y,\epsilon)\cap f_k(S(b,s))=\varnothing$.
Moreover, $x\in S(a,r)$ yields $y=\lim_{k\to\infty} f_k(x)$,
and we may assume that $f_k(x)\in B(y,\epsilon)$ for all big enough $k$. %for?all $k>k_0$.
Therefore,
\begin{equation}\label{spojitostProBod}
\deg(f_k,S(b,s),y) = \deg(f_k,S(b,s),f_k(x))
\end{equation}
for $k>k_0$.
% The uniform convergence of $f_k$ on $S(b,s)$ and the continuity of the degree
Then the continuity of the degree under uniform convergence \eqref{stability}
 yields
\begin{equation}\label{spojitostProFunkce}
\deg(f,S(b,s),y) = \lim_{k\to\infty} \deg(f_k,S(b,s),y).
\end{equation}
Because $f_k$ are homeomorphisms and $x\in (S(a,r)\setminus S(b,r)) \subset B(b,s)$ we obtain, that
$$
 \deg(f_k,S(b,s),f_k(x))\neq 0.
$$
% The degree is integer valued, and the limit of non-zero integers cannot be equal to zero.
Hence, \eqref{spojitostProBod} and \eqref{spojitostProFunkce}, as well as the fact that the degree is integer valued, give
\begin{align*}
	\deg(f,S(b,s),y) &= \lim_{k\to\infty} \deg(f_k,S(b,s),y) \\
	&= \lim_{k\to\infty} \deg(f_k,S(b,s),f_k(x)) \neq 0.
\end{align*}

It remains to prove the case when $y\notin f(S(a,r))$ (so $\deg(f,S(a,r),y)\neq 0$).
As before, the uniform convergence on spheres $S(a,r)$ and $S(b,s)$ and the continuity of the degree ensure
\begin{align*}
\deg(f,S(a,r),y) &= \lim_{m\to\infty} \deg(f_m,S(a,r),y) =\deg(f_{k},S(a,r),y),\\
\deg(f,S(b,s),y) &= \lim_{m\to\infty} \deg(f_m,S(b,s),y) =\deg(f_{k},S(b,s),y),
\end{align*}
for some big $k\in\mN$.
% The degree is integer valued, therefore there is $k\in\mN$ such that
% \begin{align*}
% \deg(f,S(a,r),y) &=\deg(f_k,S(a,r),y),\\
% \deg(f,S(b,s),y) &=\deg(f_k,S(b,s),y).
% \end{align*}
Since $y\in E(f,B(a,r))\setminus f(S(a,r))$, we have $\deg(f,S(a,r),y) \neq 0$, and so $\deg(f_k,S(a,r),y)\neq 0$.
Further, $f_k$ is a homeomorphism and $B(a,r)\subset B(b,s)$, therefore $\deg(f_k,S(a,r),y)\neq 0$ implies $\deg(f_k,S(b,s),y)\neq 0$ by \eqref{monotone}.
So $\deg(f,S(b,s),y)\neq 0$ and this completes the proof of (i).

To prove (ii) we assume, on the contrary, that $y\in f^T( B(a,r)) \cap f^T( B(b,s))$.
Then the uniform convergence and continuity of the degree ensure that there is $k\in\mN$
\begin{align*}
0\neq \deg(f,S(a,r),y) &= \lim_{m\to\infty} \deg(f_m,S(a,r),y)=\deg(f_k,S(a,r),y),\\
0\neq \deg(f,S(b,s),y) &= \lim_{m\to\infty} \deg(f_m,S(b,s),y)=\deg(f_k,S(b,s),y).
\end{align*}
Since $f_k$ is a homeomorphism, $\deg(f_k,S(a,r),y)$ and $\deg(f_k,S(b,s),y)$ cannot both differ from zero, which is a contradiction.

\end{proof}

Based on Lemma~\ref{telescopic} we follow \cite{MS} and \cite{Sv} to define the set-valued image
\begin{align*}
	f^T(a) := \bigcap_{r>0,r\notin N_a} E(f^\ast, B(a,r)).
\end{align*}
% where $N_a$ is the $\cL^1$-null set from Lemma \ref{telescopic}.
Note that $f^T(a)$ is non-empty and compact, as an intersection of a decreasing sequence of non-empty compact sets.

\begin{thm}\label{beterRepresentative}
Let $f$ be a weak limit of homeomorphisms $f_k$ in  $W^{1,p}(\Omega,\R^n)$, $p>n-1$ for $n>2$ or $p\geq 1$ for $n=2$.
Then there exists an $\cH^{n-p}$ null set $NC\subset \Omega$ and a representative $\hat{f}$ of $f$ such that $\hat{f}$ is continuous at every $x\in\Omega\setminus NC$.
Furthermore $f^T(x)$ is a singletone for every $x\in \Omega\setminus NC$, $\hat{f}=f^\ast$ $\capacity_p$-a.e.\ and $\hat{f}$ can be chosen so that $\hat{f}(x)\in f^T(x)$ for every $x\in\Omega$.
\end{thm}
\begin{proof}
Assume $p>n-1$. The theorem follows from \cite[Theorem 7.4]{MS} considering the fact that the weak limits of homeomorphisms satisfy the (INV) condition and Lemma \ref{telescopic} instead of \cite[Lemma 7.3]{MS}. Note that the condition $J_f\neq 0$ a.e.\ comes from \cite[Lemma 7.3]{MS} and plays no part in the rest of the proof.

The fact that $f^T(x)$ is a singletone follows from the proof of \cite[Theorem 7.4]{MS} as we have there
$$
NC:=\bigl\{x:\ \operatorname{diam}(f^T(x))>0\bigr\}.
$$

In the case $n=2$, $p=1$ we know that weak limit of homeomorphisms satisfy the (INV) condition thanks to the \cite[Lemma 2.6]{DPP}. And we can use the proof of \cite[Theorem 7.4]{MS} with \cite[Remark 2.9]{DPP} instad of \cite[Lemma 7.3]{MS}.
\end{proof}

\begin{proof}[Proof of the positive part of Theorem \ref{secondthm}]
This follows from Theorem \ref{beterRepresentative}. The `moreover' part with the additional assumption that $J_f>0$ a.e.\ was known before, see \cite[Lemma 3.4]{MS}. Note that this lemma holds even in the case $p=1$, $n=2$.
\end{proof}

\begin{corollary}\label{beterRepresentativeIsBetter}
The representative $\hat{f}$ from Theorem \ref{beterRepresentative} satisfies a strengthened version of condition (INV), that is for every $a\in\Omega$ and $\cL^1$-a.e.\ $r\in(0,r_a)$
\begin{enumerate}[(i)]
	\item $\hat{f}(x)\in {\hat{f}}^T(B(a,r))\cup \hat{f}(S(a,r))$ for every $x\in \overline{B(a,r)}$ and \label{beterRepresentativeIsBetter-i}
	\item $\hat{f}(x)\in\R^n\setminus{\hat{f}}(B(a,r))$ for every $x\in\Omega\setminus B(a,r)$. \label{beterRepresentativeIsBetter-ii}
\end{enumerate}
\end{corollary}
\begin{proof}
The proof follows from \cite[Corollary 7.5]{MS} with regard for Lemma~\ref{telescopic} (or \cite[Remark 2.9]{DPP} for $n=2$, $p=1$) and Theorem~\ref{beterRepresentative}.
\end{proof}

\begin{proof}[Proof of the positive part of Theorem \ref{firstthm}]
	We assume that $f=\hat{f}$, where $\hat{f}$ is from Corollary \ref{beterRepresentativeIsBetter}. Suppose, by contradiction, that there is $\delta>0$
	such that for
	$$
	F = \{y\in \R^n\colon \diam(f^{-1}(\{y\}))>0\}
	$$
	we have $\cH^{n-1+\delta}(F)>0$. Clearly, $F=\bigcup_{k\in\mN} F_k$, where
	$$F_k = \Bigl\{y\in \R^n\colon \diam(f^{-1}(\{y\}))>\frac1{k}\Bigr\}.$$
	Hence we can fix $k\in\mN$ such that $\cH^{n-1+\delta}(F_k) > 0$.

	For each $x\in \Omega$ there is a radius $r_x < \frac1{2k}$, such that $f|_{S(x,r)} \in W^{1,p}(S(x,r),\R^n)\cap \cC^0 (S(x,r),\R^n)$ (see Lemma \ref{lem:MS2.9}) and the assertion of Corollary \ref{beterRepresentativeIsBetter} holds.
	Choosing a countable covering of $\Omega$ with balls $\{B(x_i,r_{x_i})\}_{i=1}^\infty$, due to the area formula \cite[Proposition~2.7]{MS}, we know that $\cH^{n-1}(f(S(x_i,r_{x_i})))<\infty$, so $\cH^{n-1+\delta}(f(S(x_i,r_{x_i}))) = 0$.
	Therefore, even for
	$$E:=\bigcup_{i=1}^\infty f(S(x_i,r_{x_i}))$$
	we have $\cH^{n-1+\delta}(E)=0$. We now claim, that $F_k \subset E$, which is the contradiction with $\cH^{n-1+\delta}(F_k) > 0$.

	Indeed, assume that $y\in F_k\setminus E$.
	Then there must be points $z_1$ and $z_2$ in $\Omega$, such that $f(z_1)=f(z_2)=y$ and $\dist(z_1,z_2) > \frac1{k}$.
	Fix $i$ for which $z_1\in B(x_i,r_{x_i})$, $z_2\notin B(x_i,r_{x_i})$ with the  balls $B(x_i,r_{x_i})$ covering $\Omega$ and $r_{x_i}<\frac1{2k}$.
	Because $y\notin E$ we know that $y\notin S(x_i,r_{x_i})$. Therefore, Corollary \ref{beterRepresentativeIsBetter} (i) states
	$$
	y = f(z_1) \in f^T(B(x_i,r_{x_i}))
	$$
	and the assertion (ii) holds
	$$
	y = f(z_2) \in \R^n\setminus f^T(B(x_i,r_{x_i})),
	$$
	which is a contradiction.
\end{proof}

%~~~~~~~~~~~~~~~~~~~~~~~~~~~~~~~~~~~~~~~~~~~~~~~~~~~~~~~~~~~~~~~~~~~~~~~~~~~~~~

%\bibliographystyle{plain}
%\bibliography{biblio_aei}

\begin{thebibliography}{00}

\bibitem{Ball}
\by{\name{Ball}{J.}}
\paper{Global invertibility of Sobolev functions and the
interpenetration of matter}
\jour{Proc. Roy. Soc. Edinburgh Sect. A}
\vol{88\nom 3--4}
\pages{315--328}
\yr{1981}
\endpaper

\bibitem{Ball2}
\by{\name{Ball}{J.}}
\book{Some open problems in elasticity}
\publ{Geometry, mechanics, and dynamics, 3--59, Springer, New York, 2002}
\endbook

\bibitem{BK}
\by{\name{Ball}{J.}, \name{Koumatos}{K.}}
\paper{Quasiconvexity at the boundary and the nucleation of austenite}
\jour{Arch. Ration. Mech. Anal.}
\vol{219\nom 1}
\pages{89--157}
\yr{2016}
\endpaper

\bibitem{BHM}
\by{\name{Barchiesi}{M.}, \name{Henao}{D.}, \name{Mora-Corral}{C.}}
\paper{Local invertibility in Sobolev spaces with applications to nematic elastomers and magnetoelasticity}
\jour{Arch. Ration. Mech. Anal.}
\vol{224}
\pages{743--816}
\yr{2017}
\endpaper

\bibitem{BeK}
\by{\name{Bene\v{s}ov\'a}{B.}, \name{Kru\v{z}\'ik}{M.}}
\paper{Weak lower semicontinuity of integral functionals and applications}
\jour{SIAM Rev.}
\vol{59\nom 4}
\pages{703--766}
\yr{2017}
\endpaper

%Bene�ov�, Barbora; Kru��k, Martin Weak lower semicontinuity of integral functionals and applications. SIAM Rev. 59 (2017), no. 4, 703�766.

\bibitem{CHKR}
\by{\name{Campbell}{D.}, \name{Hencl}{S.}, \name{Kauranen}{A.}, \name{Radici}{E.}}
\paper{Strict limits of planar BV homeomorphisms}
\jour{Nonlinear Anal.}
\vol{177}
\pages{209--237}
\yr{2018}
\endpaper

\bibitem{CN}
\by{\name{Ciarlet}{P. G.}, \name{Ne\v{c}as}{J.}}
\paper{Injectivity and self-contact in nonlinear elasticity}
\jour{Arch. Ration. Mech. Anal.}
\vol{97}
\pages{171--188}
\yr{1987}
\endpaper


\bibitem{ConDeL2003}
\by{\name{Conti}{S.}, \name{De Lellis}{C.}}
\paper{Some remarks on the theory of elasticity for compressible
Neohookean materials}
\jour{ Ann. Sc. Norm. Super. Pisa Cl. Sci.}
\vol{2}
\pages{521--549}
\yr{2003}
\endpaper


\bibitem{DPP}
\by{\name{De Philippis}{G.} and \name{Pratelli}{A.}}
\paper{The closure of planar diffeomorphisms in Sobolev spaces}
\jour{preprint arXiv:1710.07228 }
\endprep

%\bibitem{DHMS}
%  \by{\name{D'Onofrio}{L.}, \name{Hencl}{S.}, \name{Mal\'y}{J.} and \name{Schiattarella}{R.}}
%  \paper{Note on Lusin $(N)$ condition and the distributional determinant}
%  \jour{J. Math. Anal. Appl.}
%  \vol{439}
%  \pages{171--182}
%  \yr{2016}
%  \endpaper

%\bibitem{GP}
%\by{\name{A.}{Giacomini}, \name{M.}{Ponsiglione}}
%\paper{Non-interpenetration of matter for SBV deformations of hyperelastic brittle materials}
%\jour{Proc. Roy. Soc. Edinburgh Sect. A}
%\vol{138}
%\pages{1019--1041}
%\yr{2008}
%\endpaper

%\bibitem{HMC}
%\by{\name{D.}{Henao} and \name{C.}{Mora-Corral}}
%\paper{Lusin's condition and the distributional determinant for deformations with finite energy}
%\jour{Adv. Calc. Var.}
%\vol{5}
%\pages{355--409}
%\yr{2012}
%\endpaper

\bibitem{FG}
\by{\name{Fonseca}{I.} and \name{Gangbo}{W.}}
\book{Degree Theory in Analysis and Applications}
\publ{Clarendon Press, Oxford, 1995}
\endbook

\bibitem{GP}
\by{\name{Giacomini}{A.}, \name{Ponsiglione}{M.}}
\paper{Non-interpenetration of matter for SBV deformations of hyperelastic brittle materials}
\jour{Proc. Roy. Soc. Edinburgh Sect. A}
\vol{138}
\pages{1019--1041}
\yr{2008}
\endpaper

\bibitem{GKMS}
\by{\name{Grandi}{D.}, \name{Kru\v{z}\'ik}{M.}, \name{Mainini}{E.} and \name{Stefanelli}{U.}}
\paper{A phase-field approach to Eulerian interfacial energies}
\jour{to appear in Arch. Ration. Mech. Anal, https://doi.org/10.1007/s00205-019-01391-8}
\endprep

\bibitem{GHT}
\by{\name{Guo}{C.Y.}, \name{Hencl}{S.} and \name{Tengvall}{V.}}
\paper{Mappings of finite distortion: size of the branch set}
\jour{to appear in Adv. Calc. Var}
\endprep

\bibitem{HK}
  \by{\name{Hencl}{S.} and \name{Koskela}{P.}}
  \book{Lectures on Mappings of finite distortion}
  \publ{Lecture Notes in Mathematics 2096, Springer, 2014, 176pp}
  \endbook

\bibitem{HO}
\by{\name{Hencl}{S.} and \name{Onninen}{J.}}
\paper{Jacobian of weak limits of Sobolev homeomorphisms}
\jour{Adv. Calc. Var.}
\vol{11 \nom 1}
\pages{65--73}
\yr{2018}
\endpaper

\bibitem{HR}
\by{\name{Hencl}{S.} and  \name{Rajala}{K.}}
\paper{Optimal assumptions for discreteness}
\jour{Arch. Ration. Mech. Anal}
\vol{207 \nom 3}
\pages{775--783}
\yr{2013}
\endpaper


 \bibitem{HMC}
\by{\name{Henao}{D.} and \name{Mora-Corral}{C.}}
\paper{Invertibility and weak continuity of the determinant for the modelling of cavitation and fracture in nonlinear elasticity}
\jour{Arch. Ration. Mech. Anal.}
\vol{197}
\pages{619--655}
\yr{2010}
\endpaper

\bibitem{IO}
\by{\name{Iwaniec}{T.} and \name{Onninen}{J.}}
\paper{Monotone Sobolev Mappings of Planar Domains and Surfaces}
\jour{Arch. Ration. Mech. Anal.}
\vol{219}
\pages{159--181}
\yr{2016}
\endpaper

\bibitem{IO2}
\by{\name{Iwaniec}{T.} and \name{Onninen}{J.}}
\paper{Limits of Sobolev homeomorphisms}
\jour{J. Eur. Math. Soc.}
\vol{19}
\pages{473--505}
\yr{2017}
\endpaper

\bibitem{IS}
\by{\name{Iwaniec}{T.} and \name{\v Sver\'ak}{V.}}
\paper{On mappings with integrable dilatation}
\jour{Proc. Amer. Math. Soc.}
\vol{118}
\pages{181--188}
\yr{1993}
\endpaper

\bibitem{KroVal2019}
\by{\name{Kr\"omer}{S.} and \name{Valdman}{J.}}
\paper{Global injectivity in second-gradient nonlinear elasticity and its approximation with penalty terms}
\jour{Math. Mech. Solids}
\vol{Online first}
\pages{}
\yr{2019}
\endpaper

\bibitem{MaVi}
\by{\name{Manfredi}{J.} and \name{Villamor}{E.}}
\paper{An extension of Reshetnyak's theorem}
\jour{Indiana Univ. Math. J.}
\vol{47\nom 3}
\pages{1131--1145}
\yr{1998}
\endpaper

\bibitem{MieRou2016}
\by{\name{Mielke}{A.} and \name{Roub\'i\v cek}{T.}}
\paper{ Rate-independent elastoplasticity at finite strains and its numerical approximation}
\jour{Math. Models Methods Appl. Sci.}
\vol{26\nom 12}
\pages{2203--2236}
\yr{2016}
\endpaper

\bibitem{MV}
\by{\name{Molchanova}{A. O.} and \name{Vodop'yanov}{S. K.}}
\paper{Injectivity almost everywhere and mappings with finite distortion in nonlinear elasticity}
\jour{preprint arXiv:1704.08022}
\endprep

\bibitem{MS}
\by{\name{M\"uller}{S.} and \name{Spector}{S. J.}}
\paper{An existence theory for nonlinear elasticity that
allows for cavitation}
\jour{Arch. Ration. Mech. Anal.}
\vol {131, \rm no. 1}
\yr {1995}
\pages {1--66}
\endpaper


\bibitem{MST}
\by{\name{M\"uller}{S.}, \name{Spector}{S. J.} and \name{Tang}{Q.}}
\paper{Invertibility and a topological property of Sobolev maps}
\jour{SIAM J. Math. Anal.}
\vol{27}
\yr{1996}
\pages{959--976}
\endpaper

\bibitem{MTY}
\by {\name{M\"uller}{S.}, \name{Tang}{Q.} and \name{Yan}{B. S.}}
\paper{On a new class of elastic deformations not allowing for
cavitation}
\jour{Analyse Nonlineaire}
\vol{11}
\yr{1994}
\pages{217--243}
\endpaper



\bibitem{Sv}
\by{\name{\v Sver\'ak}{V.}}
\paper{Regularity properties of deformations with finite energy}
\jour{Arch. Ration. Mech. Anal.}
\vol{100\nom 2}
\pages{105--127}
\yr{1988}
\endpaper


\bibitem{SwaZie2002}
\by{\name{Swanson}{D.}, \name{Ziemer}{W. P.}}
\paper{A topological aspect of Sobolev mappings}
\jour{Calc. Var. Partial Differential Equations}
\vol{14\nom 1}
\pages{69--84}
\yr{2002}
\endpaper



\bibitem{SwaZie2004}
\by{\name{Swanson}{D.}, \name{Ziemer}{W. P.}}
\paper{The image of a weakly differentiable mapping}
\jour{SIAM J. Math. Anal.}
\vol{35\nom 5}
\pages{1099--1109}
\yr{2004}
\endpaper


\bibitem{T}
\by{\name{Tang}{Q.}}
\paper{Almost-everywhere injectivity in nonlinear elasticity}
\jour{Proc. Roy. Soc. Edinburgh Sect. A}
\vol{109}
\pages{79--95}
\yr{1988}
\endpaper

\bibitem{Z}
\by{\name{Ziemer}{W. P.}}
\book{Weakly differentiable functions. Graduate texts in Mathematics, 120}
\publ{Springer-Verlag}
\yr{1989}
\endbook



\end{thebibliography}

\end{document}